\documentclass[reqno,12pt]{amsart}

\usepackage[margin=1in]{geometry}%marges
\usepackage{xspace}
\usepackage[utf8]{inputenc}
\usepackage{graphicx}
\usepackage{amssymb}
\usepackage{mathrsfs}
\usepackage[active]{srcltx}
\usepackage{url}
\usepackage{hyperref}
\usepackage{amsmath,amstext,amsxtra,amsgen,amsbsy,amsopn,amscd,amsthm,amsfonts}
\usepackage{latexsym}
\usepackage{dsfont}
\usepackage{enumitem}%pouvoir enlever l'indentation des itemize

\newtheorem{theorem}{Theorem}

\newtheorem{proposition}{Proposition}
\newtheorem{lemma}[proposition]{Lemma}
\newtheorem{corollary}[proposition]{Corollary}

\theoremstyle{definition}

\theoremstyle{remark}

\newtheorem{remark}[proposition]{Remark}

\newcommand\R{{\ensuremath {\mathbb R} }}
\newcommand\C{{\ensuremath {\mathbb C} }}
\newcommand\N{{\ensuremath {\mathbb N} }}

\newcommand\T{{\ensuremath {\mathbb T} }}

\renewcommand\phi{\varphi}

\renewcommand\le{\leqslant}
\renewcommand\ge{\geqslant}
\renewcommand\epsilon{\varepsilon}
\renewcommand\hat{\widehat}
\renewcommand\tilde{\widetilde}
\renewcommand\bar{\overline}
\newcommand{\gH}{\mathfrak{H}}
\newcommand{\gK}{\mathfrak{K}}

\newcommand{\gS}{\mathfrak{S}}

\newcommand\ii{{\ensuremath {\infty}}}

\newcommand{\norm}[1]{ \left| \! \left| #1 \right| \! \right| }

\newcommand{\cE}{\mathcal{E}}

\newcommand{\cO}{\mathcal{O}}

\newcommand{\cV}{\mathcal{V}}

\newcommand\1{{\ensuremath {\mathds 1} }}
\newcommand{\Sph}{\mathbb{S}}

\newcommand\cS{\mathcal{S}}

\DeclareMathOperator{\re}{Re}

\DeclareMathOperator{\supp}{supp}

\DeclareMathOperator{\Tr}{Tr}
\DeclareMathOperator{\tr}{Tr}

\date{August 29, 2016}

\title[Spectral clusters and Oscillatory integrals]{Spectral cluster bounds for orthonormal systems\\ and oscillatory integral operators in Schatten spaces}

\author{Rupert L. Frank}
\address{Rupert L. Frank, Mathematics 253-37, Caltech, Pasadena, CA 91125, USA}
\email{rlfrank@caltech.edu}

\author{Julien Sabin}
\address{Julien Sabin, Laboratoire de Math\'ematiques d'Orsay, Univ. Paris-Sud, CNRS, Universit\'e
Paris-Saclay, 91405 Orsay, France.}
\email{Julien.Sabin@math.u-psud.fr}

\begin{document}

\begin{abstract}
We generalize the $L^p$ spectral cluster bounds of Sogge for the Laplace--Beltrami operator on compact Riemannian manifolds to systems of orthonormal functions. The optimality of these new bounds is also discussed. These spectral cluster bounds follow from Schatten-type bounds on oscillatory integral operators.
\end{abstract}

\maketitle

%%%%%%%%%%%%%%%%%%%%%%%%%%%%%%%%%%%%%%
%%%%%%%%%%%%%%%%%%%%%%%%%%%%%%%%%%%%%%
\section*{Introduction}
%%%%%%%%%%%%%%%%%%%%%%%%%%%%%%%%%%%%%%
%%%%%%%%%%%%%%%%%%%%%%%%%%%%%%%%%%%%%%

In this paper we are interested in concentration properties of orthonormal systems of eigenfunctions or quasi-modes corresponding to large eigenvalues of the Laplace--Beltrami operator on a manifold. Let $(M,g)$ be a smooth, compact Riemannian manifold without boundary of dimension $N\ge2$. We denote by $\Delta_g$ the Laplace--Beltrami operator on $M$, which is a self-adjoint, non-negative operator in $L^2(M)$, defined with respect to the Riemannian volume measure $dv_g$ on $M$. We emphasize that we use the geometric, rather than the analytic sign convention for the Laplacian. For any $\lambda\ge0$, we define the spectral projector
\begin{equation}
\label{eq:pi}
\Pi_\lambda :=\1(\lambda^2\le\Delta_g<(\lambda+1)^2)
\end{equation}
and the spectral cluster
$$E_\lambda:=\Pi_\lambda L^2(M).$$
(The upper bound $(\lambda+1)^2$ can be replaced by $(\lambda+C)^2$ for any fixed constant $C>0$, without changing the following results qualitatively.)

Let $Q\subset E_\lambda$ be a subspace and let $(f_j)_{j\in J}$ be an orthonormal basis in $Q$. Then
$$
\rho^Q:= \sum_{j\in J} |f_j|^2
$$
is independent of the choice of the basis and our goal is to obtain bounds on the $L^{p/2}(M)$ norm of $\rho^Q$ for $2\le p\le\infty$ in terms of $\lambda$ and $|J|=\dim Q$. Note that for $p=2$ we have $\|\rho^Q\|_{L^1(M)} = \dim Q$. For $p>2$ the quotient $\|\rho^Q\|_{L^{p/2}(M)}/\|\rho^Q\|_{L^1(M)}$ quantifies concentration properties of functions in $Q$ in some averaged sense; see, for instance, Remark \ref{pqr} below.

The two extreme cases $Q=E_\lambda$ and $\dim Q=1$ have been studied in detail and are classical results in semi-classical analysis and spectral geometry. We recall the optimal remainder estimate in Weyl's law for the eigenvalues of $\Delta_g$, due to Avakumovi\'c \cite{Avakumovic-56} and Levitan \cite{Levitan-52} and vastly generalized by H\"ormander \cite[Thm. 1.1]{Hormander-68}, which says that
\begin{equation}\label{eq:weylsharp}
    \tr\1(\Delta_g<\lambda^2)=(2\pi)^{-N}|\{(x,\xi)\in T^*M,\ g_x(\xi,\xi)\le1\}|\ \lambda^N+\cO_{\lambda\to+\ii}(\lambda^{N-1}).
\end{equation}
The usual proof of these asymptotics (see, e.g., \cite[Lem. 4.3]{Hormander-68} or \cite[Sec. 4.2]{Sogge-book}) proceeds by first showing that
\begin{equation}
\label{eq:weylbound}
\norm{\rho^{E_\lambda}}_{L^{p/2}(M)} \le C \lambda^{N-1}
\qquad\text{if}\ 2\le p\le\infty \,,
\end{equation}
with some $C>0$ independent of $\lambda\ge 1$. By H\"older's inequality, this bound with $p=\infty$ implies a similar bound for any $2\le p\le \infty$. We also note the elementary fact that $\limsup_{\lambda\to\infty} \lambda^{-N} \tr\1(\Delta_g<\lambda^2)>0$ implies that
\begin{equation}\label{eq:lower-bound-dim-cluster}
    \limsup_{\lambda\to+\ii}\lambda^{-(N-1)}\norm{\rho^{E_\lambda}}_{L^1(M)}=\limsup_{\lambda\to+\ii}\lambda^{-(N-1)}\dim E_\lambda>0.
\end{equation}
Therefore, for any $2\le p\le\infty$ the power of $\lambda$ in the bound \eqref{eq:weylbound} cannot be decreased. Since all $L^{p/2}$ norms of $\rho^{E_\lambda}$ are of the same order, we interpret \eqref{eq:weylbound} and \eqref{eq:lower-bound-dim-cluster} as a \emph{non}-concentration property of $\rho^{E_\lambda}$.

Bounds in the other extreme case $\dim Q=1$ are a celebrated result of Sogge \cite[Thm. 2.2]{Sogge-88} (see also \cite[Thm. 5.1.1]{Sogge-book}). Namely, for any $f\in E_\lambda$ we have
\begin{equation}\label{eq:sogge-bound}
\norm{f}_{L^p(M)}\le C\lambda^{s(p)}\norm{f}_{L^2(M)},
\end{equation}
with some $C>0$ independent of $f$ and $\lambda\ge 1$, where
$$s(p)=\begin{cases}
        N\left(\frac12-\frac1p\right)-\frac12 & \text{if}\ \frac{2(N+1)}{N-1}\le p\le +\ii \,,\\
        \frac{N-1}{2}\left(\frac12-\frac1p\right) & \text{if}\ 2\le p\le \frac{2(N+1)}{N-1}\,.
       \end{cases}
$$
For any $M$ and for any $p$, the power $s(p)$ is sharp in the sense that there exist $f_\lambda\in E_\lambda$ with $\norm{f_\lambda}_{L^p}/\norm{f_\lambda}_{L^2}\sim\lambda^{s(p)}$ as $\lambda\to+\ii$. The quotient $\norm{f}_{L^p}/\norm{f}_{L^2}$ measures in some sense the ``concentration'' of the function $f$, hence Sogge's result may be seen as an optimal concentration estimate for functions in spectral clusters of the Laplacian. The fact that different $L^{p/2}$ norms grow differently with $\lambda$ for different $p$'s means that there is a concentration phenomenon and the piecewise definition of $s(p)$ reflects the fact that there are two competing concentration mechanisms, which will also become relevant for us later on.

Our main result in this paper is a bound which interpolates in an optimal way between the two extreme cases $Q=E_\lambda$ and $\dim Q=1$. We shall show that (see Theorem \ref{thm:main})
\begin{equation}
\label{eq:mainintro}
\left\|\rho^Q\right\|_{L^{p/2}(M)} \le C \lambda^{2s(p)} (\dim Q)^{1/\alpha(p)}
\end{equation}
with some $C$ independent of $Q$ and $\lambda\ge 1$, where
$$
\alpha(p) =\begin{cases}
\frac{p(N-1)}{2N} & \text{if}\ \frac{2(N+1)}{N-1}\le p\le +\ii,\\
\frac{2p}{p+2} & \text{if}\ 2\le p\le \frac{2(N+1)}{N-1} \,.
       \end{cases}
$$
We emphasize that \eqref{eq:mainintro} coincides with Sogge's bound \eqref{eq:sogge-bound} for $\dim Q=1$ and with the sharp Weyl-law bound \eqref{eq:weylbound} for $Q=E_\lambda$ (recalling also \eqref{eq:lower-bound-dim-cluster}).

At least for $N=2$ and $M=\Sph^2$ our bound is optimal in the following strong sense (see Theorem \ref{thm:optimality}). For any $2\le p\le\infty$ and any $r_\lambda$ with $1\ll r_\lambda\ll\dim E_\lambda$ there is a subspace $Q_\lambda\subset E_\lambda$ with $\dim Q_\lambda \sim r_\lambda$ and, for some $c>0$,
\begin{equation}
\label{eq:optimalityintro}
\left\|\rho^{Q_\lambda} \right\|_{L^{p/2}(M)} \ge c \lambda^{2s(p)} (\dim Q_\lambda)^{1/\alpha(p)} \,.
\end{equation}

The crucial point in our bound \eqref{eq:mainintro} is that the exponent $\alpha(p)>1$. In fact, applying the triangle inequality in the definition of $\rho^Q$ and estimating each function $f_j$ using Sogge's bound \eqref{eq:sogge-bound} we obtain
\begin{equation}
\label{eq:mainintrotriangle}
\left\|\rho^Q\right\|_{L^{p/2}(M)} \le C \lambda^{2s(p)} \dim Q \,,
\end{equation}
which, however, is not optimal. Therefore the crucial point which leads to the decrease from $1$ to $1/\alpha(p)$ and which has been ignored in the `triangle inequality' derivation of \eqref{eq:mainintrotriangle} is the \emph{orthogonality of the functions $f_j$}.

The observation that orthonormality improves the dependence on the number of functions as compared to a simple use of the triangle inequality was originally made in \cite{LieThi-75,LieThi-76,Lieb-83a} in the context of the Sobolev inequality and was recently extended to the Strichartz inequality in \cite{FraLewLieSei-13,FraSab-14}. Here we will develop our method from \cite{FraSab-14} further to prove \eqref{eq:mainintro}. In particular, as was noticed by Sogge, $L^p$-bounds on spectral clusters can be reduced to estimates on oscillatory integral operators. A first step in our proof of \eqref{eq:mainintro} is thus to prove bounds on oscillatory integral operators for systems of orthonormal functions. Since oscillatory integral operators appear in other contexts as is discussed below, our new results on such operators may have other applications as well, for instance, in relation to resolvent bounds (as was done for the resolvent of the Laplacian in $\R^N$ in \cite{FraSab-14}). We will discuss these bounds and their history in detail in the next section.

The proof of our optimality result \eqref{eq:optimalityintro} is rather involved and uses both WKB methods and facts about spherical harmonics. Essentially, we are dealing with WKB approximations of sums of squares of quasi-modes (after separation of variables, the functions are  no longer eigenfunctions of the same one-dimensional operator) and the major difficulty is to control their oscillations. We hope that the techniques that we develop in the proof will be relevant to related problems in mathematical physics and many-body quantum mechanics.

Our work raises the following open questions, which we think might be worth further investigation. First, while on spheres $M=\Sph^{N-1}$ the distinction between eigenspaces and spectral clusters disappears, we would like to emphasize that the functions $f_j$ building up $\rho^Q$ may be \emph{sums} of eigenfunctions corresponding to different eigenvalues (between $\lambda^2$ and $(\lambda+1)^2$). Restricting to eigenspaces instead of spectral clusters might lead to improved bounds, depending on the manifold. For instance, for $M=\T^2$, \eqref{eq:sogge-bound} can be improved for eigenfunctions with $p=4$, and also \eqref{eq:weylsharp} can be improved. Another question concerns our optimality construction \eqref{eq:optimalityintro}. It would be interesting to prove optimality for general manifolds in the spirit of the original work of Sogge.

This article is organized as follows. In Section \ref{sec:carleson}, we prove bounds on oscillatory integral operators for orthonormal functions. In Section \ref{sec:clusters}, we apply these results to prove the generalization \eqref{eq:mainintro} of Sogge's bounds to orthonormal functions, and discuss further optimality aspects of our bound. Finally, in Section \ref{sec:optimality}, we prove the optimality on $M=\Sph^2$ for a fixed number of functions. 

\subsection*{Acknowledgement} Partial support by U.S. National Science Foundation DMS-1363432 (R.L.F.) is acknowledged. J.S. would like to thank Nicolas Burq and Bernard Helffer for useful discussions. 

%%%%%%%%%%%%%%%%%%%%%%%%%%%%%%%%%%%%%%%%%%%%%%%%%
\section{Schatten bounds on oscillatory integral operators}\label{sec:carleson}
%%%%%%%%%%%%%%%%%%%%%%%%%%%%%%%%%%%%%%%%%%%%%%%%%

An operator $T_\lambda$ acting on functions $f\in L^1(\R^N)$ by the following relation
\begin{equation}
 T_\lambda f(x)=\int_{\R^N}e^{i\lambda\psi(x,y)}a(x,y)f(y)\,dy,\ \forall x\in\R^N,
\end{equation}
is called an \emph{oscillatory integral operator}. Here, the function $a\in C^\ii_0(\R^N\times\R^N)$ is the amplitude, $\psi\in C^\ii(\R^N\times\R^N)$ is the phase function, and $\lambda>0$ is a parameter which typically becomes large. The problem of understanding the behavior of the norm $\norm{T_\lambda}_{L^p\to L^q}$ as $\lambda\to+\ii$ is a central result in harmonic analysis and we refer, for instance, to \cite[Sec. IX]{Stein-book-93} and \cite[Ch. 2]{Sogge-book} for background information.

The following theorem generalizes existing bounds on oscillatory integral operators to systems of orthonormal functions.

\begin{theorem}\label{thm:main-lemma}
Let $N\ge2$, $a\in C^\ii_0(\R^N\times\R^{N-1})$, and $\psi\in C^\ii(\R^N\times\R^{N-1})$ satisfying
\begin{equation}
\label{eq:main-lemma-ass1}
{\rm Rank}\left(\frac{\partial^2\psi}{\partial x_i\partial y_j}(x,y)\right)_{\substack{1\le i\le N \\ 1\le j\le N-1}}=N-1,\ \forall (x,y)\in\supp a.
\end{equation}
This assumption implies that for any $(x_0,y_0)\in\supp a$ there is a neighborhood $V$ of $y_0$ in $\R^{N-1}$ such that $\nabla_x\psi(x_0,V)=:S_{(x_0,y_0)}\subset\R^N$ is a hypersurface, and we assume in addition that for all $(x_0,y_0)\in{\rm supp}(a)$,
\begin{equation}
\label{eq:main-lemma-ass2}
\textrm{the surface}\ S_{(x_0,y_0)}\ \textrm{has non-zero Gauss curvature at}\ \nabla_x\psi(x_0,y_0). 
\end{equation}
For all $f\in L^1(\R^{N-1})$, define
 $$T_\lambda f(x)=\int_{\R^{N-1}}e^{i\lambda\psi(x,y)}a(x,y)f(y)\,dy,\ \forall x\in\R^N.$$
Then, there exists $C>0$ such that for any $\lambda\ge1$, for any orthonormal system $(f_j)_{j\in J}\subset L^2(\R^{N-1})$ and any coefficients $(\nu_j)_{j\in J}\subset\C$, we have
 $$\norm{\sum_{j\in J}\nu_j|T_\lambda f_j|^2}_{L^\frac{N+1}{N-1}(\R^N)}\le C\lambda^{-\frac{N(N-1)}{(N+1)}}\left(\sum_{j\in J}|\nu_j|^{1+\frac{1}{N}}\right)^\frac{N}{N+1}.$$
\end{theorem}

Theorem \ref{thm:main-lemma} for a single function was proved by Carleson--Sj\"olin \cite[Thm. I]{CarSjo-72} when $N=2$, with a simpler proof by H\"ormander \cite[Thm. 1.2]{Hormander-73}, and by Stein \cite[Thm. 10]{Stein-book-84} when $N\ge3$. These authors were motivated by the restriction problem for the Fourier transform (see, e.g., \cite{Stein-book-93,Sogge-book}). In particular, the single-function version of Theorem \ref{thm:main-lemma} implies the Stein--Tomas restriction theorem \cite[Thm. 3]{Stein-book-84}, \cite{Tomas-75}, and our multi-function generalization of it leads to a multi-function generalization of the Stein--Tomas restriction theorem. We have proved this generalization in \cite[Thm. 4]{FraSab-14} using a different, less general method.

\begin{remark}\label{st}
We briefly explain how to recover \cite[Thm. 4]{FraSab-14} from Theorem \ref{thm:main-lemma}. Let $S\subset\R^N$ be a compact surface with non-zero Gauss curvature. Locally, we can write it as the graph of some function $h$. We thus apply Theorem \ref{thm:main-lemma} with $\psi(x,y)= x \cdot (y,h(y))$, which satisfies the non-degeneracy and the curvature conditions, and with $a(x,y)=\beta'(x)\beta(y)$, where $\beta$ is some localization function. Changing variables $x'=\lambda x$ and taking the limit $\lambda\to+\ii$ leads to the Stein--Tomas restriction theorem from \cite{FraSab-14}.
\end{remark}

\begin{remark}
Analytically, assumption \eqref{eq:main-lemma-ass2} means the following: It follows from \eqref{eq:main-lemma-ass1} that for any $(x_0,y_0)\in\supp a$ there is a (unique up to sign) vector $e\in\Sph^{N-1}$ so that $y\mapsto e\cdot\nabla_x\psi(x_0,y)$ has a critical point at $y=y_0$, and then \eqref{eq:main-lemma-ass2} says that
\begin{equation*}
\det\left( \frac{\partial^2}{\partial y_i\partial y_j} e\cdot\nabla_x\psi(x_0,y_0) \right)_{\substack{1\le i\le N-1 \\ 1\le j\le N-1}} \neq 0.
\end{equation*}
\end{remark}

Once Theorem \ref{thm:main-lemma} is known, it is not difficult to obtain a result applicable to the phase function $\psi(x,y)\sim|x-y|$. It is actually this corollary which will be used in the next section about spectral clusters.

\begin{corollary}\label{coro}
 Let $N\ge2$, $a\in C^\ii_0(\R^N\times\R^N)$, and $\psi\in C^\ii(\R^N\times\R^N)$ satisfying
 $${\rm Rank}\left(\frac{\partial^2\psi}{\partial x_i\partial y_j}(x,y)\right)_{1\le i,j\le N}=N-1,\ \forall (x,y)\in{\rm supp}(a).$$
As a result, for any $(x_0,y_0)\in\supp a$ there is a (unique up to sign) vector $e\in\Sph^{N-1}$ so that $x\mapsto e\cdot\nabla_x\psi(x_0,y)$ has a critical point at $y=y_0$, and our further assumption is that

 For any $(x_0,y_0)\in{\rm supp}(a)$, the previous assumptions implies that there exists a neighborhood $V$ of $y_0$ in $\R^N$ such that $\nabla_x\psi(x_0,V)=:S_{(x_0,y_0)}$ is a hypersurface in $\R^N$. Assume furthermore that for all $(x_0,y_0)\in{\rm supp}(a)$, the surface $S_{(x_0,y_0)}$ has non-zero Gauss curvature at $\nabla_x\psi(x_0,y_0)$. 

 For all $f\in L^1(\R^{N})$, define
 $$T_\lambda f(x)=\int_{\R^N}e^{i\lambda\psi(x,y)}a(x,y)f(y)\,dy,\ \forall x\in\R^N.$$
Then, there exists $C>0$ such that for any $\lambda\ge1$, for any orthonormal system $(f_j)\subset L^2(\R^N)$ and any set of coefficients $(\nu_j)\subset\C$, we have
  $$\norm{\sum_{j}\nu_j|T_\lambda f_j|^2}_{L^\frac{N+1}{N-1}(\R^N)}\le C\lambda^{-\frac{N(N-1)}{(N+1)}}\left(\sum_{j}|\nu_j|^{1+\frac{1}{N}}\right)^\frac{N}{N+1}.$$ 
\end{corollary}

This corollary for a single function can be found, for instance, in \cite[Cor. 2.2.3]{Sogge-book}. It can be used to study the Bochner--Riesz problem \cite[Sec. IX.2.2]{Stein-book-93}, \cite[Sec. 2.2.3]{Sogge-book}, to obtain $L^p$ bounds on spectral clusters of the Laplace--Beltrami operator on compact Riemannian manifolds \cite[Sec. 5.5.1]{Sogge-book} or resolvent estimates for this operator \cite{Sogge-88,DosKenSal-14}.

\begin{remark}
The crucial point of Theorem \ref{thm:main-lemma} and Corollary \ref{coro} is the quantity $(\sum |\nu_j|^{1+1/N})^{N/(N+1)}$ on the right side. If we would only use the single function versions of these theorems and the triangle inequality, we would only get the larger quantity $\sum |\nu_j|$ on the right side. This gain is due to the orthogonality of the functions $f_j$. Our optimality results in this paper and in \cite{FraSab-14} (for instance, in the Stein--Tomas context) show that the bounds do \emph{not} hold with $(\sum |\nu_j|^\alpha)^{1/\alpha}$ for some $\alpha>1+1/N$.
\end{remark}

Following \cite{FraSab-14} we will rephrase Theorem \ref{thm:main-lemma} and Corollary \ref{coro} in terms of trace ideal properties of certain compact operators. To do so, recall that for two Hilbert spaces $\gH$, $\gK$, the Schatten class $\gS^\alpha(\gH,\gK)$ for $\alpha>0$ is defined as the set of all compact linear operators $A:\gH\to\gK$ such that 
 $\tr(A^*A)^{\alpha/2}<\ii$. For such an operator $A$, its Schatten norm is defined as 
 $$\norm{A}_{\gS^\alpha(\gH,\gK)}:=\left(\tr(A^*A)^{\alpha/2}\right)^{1/\alpha}.$$
When $\gH=\gK$, we write $\gS^\alpha(\gH)$ instead of $\gS^\alpha(\gH,\gH)$. For background on Schatten spaces we refer, for instance, to \cite{Sim}.

\begin{remark}\label{rk:schatten}
 By \cite[Prop. 1]{FraSab-14}, Theorem \ref{thm:main-lemma} is equivalent to the estimate
 \begin{equation}\label{eq:WT-schatten}
   \norm{WT_\lambda}_{\gS^{2(N+1)}(L^2(\R^{N-1}),L^2(\R^N))}\le C\lambda^{-\frac{N(N-1)}{2(N+1)}}\norm{W}_{L^{N+1}(\R^N)},\ \forall W\in L^{N+1}(\R^N)
 \end{equation}
and Corollary \ref{coro} is equivalent to the estimate
 \begin{equation}\label{eq:coro-schatten}
  \norm{WT_\lambda}_{\gS^{2(N+1)}(L^2(\R^N))}\le C\lambda^{-\frac{N(N-1)}{2(N+1)}}\norm{W}_{L^{N+1}(\R^N)},\ \forall W\in L^{N+1}(\R^N),
 \end{equation}
In fact, these are the formulations that we will prove.
\end{remark}

\begin{proof}[Proof of Theorem \ref{thm:main-lemma}]
According to Remark \ref{rk:schatten} we may prove \eqref{eq:WT-schatten}. In the proof of \cite[Thm. IX.1]{Stein-book-93}, Stein introduces an analytic family of operators $(U_s)$ with $-(N-1)/2\le\text{Re}\, s\le1$, mapping functions on $\R^N$ to functions on $\R^N$, such that $U^0=T_\lambda T_\lambda^*$ and satisfying the bounds
$$\begin{cases}
   \norm{U^{1+it}}_{L^2(\R^N)\to L^2(\R^N)}\le C\lambda^{-N}, \\
   \norm{U^{-(N-1)/2+it}}_{L^1(\R^N)\to L^\ii(\R^N)}\le C,
  \end{cases}
$$
for all $t\in\R$ and $\lambda\ge1$, and for some $C>0$ independent of $t$ and $\lambda$. These two bounds imply that for any functions $W_1,W_2$,
$$\begin{cases}
   \norm{W_1U^{1+it}W_2}_{L^2(\R^N)\to L^2(\R^N)}\le C\lambda^{-N}\norm{W_1}_{L^\ii(\R^N)}\norm{W_2}_{L^\ii(\R^N)}, \\
   \norm{W_1U^{-(N-1)/2+it}W_2}_{\gS^2(L^2(\R^N))}\le C\norm{W_1}_{L^2(\R^N)}\norm{W_2}_{L^2(\R^N)},
  \end{cases}
$$
where in the second estimate we used that for any operator $A$ acting on functions on $\R^N$,
$$\norm{A}_{\gS^2(L^2(\R^N))}^2=\int_{\R^N}\int_{\R^N}|A(x,y)|^2\,dx\,dy,\qquad\norm{A}_{L^1(\R^N)\to L^\ii(\R^N)}=\norm{A(\cdot,\cdot)}_{L^\ii(\R^N\times\R^N)}.$$
Here $A(\cdot,\cdot)$ denotes the integral kernel of an integral operator $A$. Using complex interpolation between these two bounds as in \cite[Prop. 1]{FraSab-14}, we deduce that 
$$\norm{WU_0\bar{W}}_{\gS^{N+1}(L^2(\R^N))}\le C\lambda^{-\frac{N(N-1)}{N+1}}\norm{W}_{L^{N+1}(\R^N)}^2.$$
Since
$$\norm{WT_\lambda}_{\gS^{2(N+1)}(L^2(\R^{N-1}),L^2(\R^N))}^2=\norm{WT_\lambda T_\lambda^*\bar{W}}_{\gS^{N+1}(L^2(\R^N))},$$
we have proved \eqref{eq:WT-schatten}.
\end{proof}

\begin{proof}[Proof of Corollary \ref{coro}]
According to Remark \ref{rk:schatten} we may prove \eqref{eq:coro-schatten}, and to do so, we follow the arguments of \cite[Cor. 2.2.3]{Sogge-book}. Define the matrix
$$M(x,y):=\left(\frac{\partial^2\psi}{\partial x_i\partial y_j}(x,y)\right)_{1\le i,j\le N}$$
and let $(x_0,y_0)\in\supp(a)$.
Since the rank of $M(x_0,y_0)$ is $N-1$, there exists $1\le j_0\le N$ such that the matrix $M(x_0,y_0)$ with the $j_0^{th}$ column removed has maximal rank $N-1$. By continuity, there is a neighborhood $\cV$ of $(x_0,y_0)$ such that for all $(x,y)\in \cV$, $M(x,y)$ also has maximal rank $N-1$ when the $j_0^{th}$ column is removed. By compactness of $\text{supp}\,(a)$, we may cover $\text{supp}\,(a)$ by a finite number of such neighborhoods $(\cV_k)$. If $(\phi_k)$ is a partition of unity subordinated to $(\cV_k)$, we define $T_\lambda^{(k)}$ to be the oscillatory integral operator with phase $\psi$ and amplitude $a\phi_k$. Then, using $T_\lambda=\sum_k T_\lambda^{(k)}$ and hence
$$\norm{WT_\lambda}_{\gS^{2(N+1)}(L^2(\R^N))}\le\sum_k\norm{WT_\lambda^{(k)}}_{\gS^{2(N+1)}(L^2(\R^N))},$$
it is enough to estimate a single $WT_\lambda^{(k)}$. Up to exchanging coordinates, we may assume that $j_0=N$ and we write $y\in\R^N$ as $y=(y',t)$ with $y'\in\R^{N-1}$ and $t\in\R$. Up to reducing $\cV$ if necessary, we may also assume that $\cV$ is a product of neighborhoods $\cV=\cV_{x_0}\times \cV_{y_0'}\times \cV_{t_0}\subset\R^N\times\R^{N-1}\times\R$. For any $(x,y',t)\in\cV$, the image of the map $y''\mapsto\nabla_x\psi(x,y'',t)$ for $y''$ in a neighborhood of $y'$ is a surface which is a portion of $S_{(x,y',t)}$ containing a neighborhood of $\nabla_x\psi(x,y',t)$ in $S_{(x,y',t)}$. In particular, it also has non-vanishing Gauss curvature at $\nabla_x\psi(x,y',t)$, implying that for any fixed $t\in \cV_{t_0}$, the phase function $\psi(\cdot,\cdot,t)$ satisfies the assumptions of Theorem \ref{thm:main-lemma}. The operator $T_{\lambda,t}$ defined as 
$$T_{\lambda,t}h(x)=\int_{\R^{N-1}}e^{i\lambda\psi(x,y',t)}a(x,y',t)h(y')\,dy',\ \forall x\in\R^N,\ \forall h\in L^2(\R^{N-1}),$$
thus satisfies the estimate of Theorem \ref{thm:main-lemma}. We may write $T_\lambda$ as 
$$T_\lambda f(x)=\int_\R T_{\lambda,t}[f(\cdot,t)](x)\,dt,\ \forall x\in\R^N,\ \forall f\in L^2(\R^N),$$
implying that 
$$T_\lambda T_\lambda^*=\int_\R T_{\lambda,t}T_{\lambda,t}^*\,dt=\int_{\cV_{t_0}}T_{\lambda,t}T_{\lambda,t}^*\,dt.$$
We deduce from this representation and from Theorem \ref{thm:main-lemma} that
$$\norm{W T_\lambda T_\lambda^*\bar{W}}_{\gS^{N+1}(L^2(\R^N))}\le|\cV_{t_0}|\sup_{t\in \cV_{t_0}}\norm{W T_{\lambda,t}T_{\lambda,t}^*\bar{W}}_{\gS^{N+1}(L^2(\R^N))}\le C\lambda^{-\frac{N(N-1)}{N+1}}\norm{W}_{L^{N+1}(\R^N)}^2,$$
provided that the constant $C>0$ appearing in Theorem \ref{thm:main-lemma} is uniform in the parameter $t\in \cV_{t_0}$, which is the case since it is uniform as soon as there is a positive lower bound on the Gauss curvatures of the surfaces involved.
\end{proof}

%%%%%%%%%%%%%%%%%%%%%%%%%%%%%%%%%%%%%%%%%%
\section{Spectral cluster bounds}\label{sec:clusters}
%%%%%%%%%%%%%%%%%%%%%%%%%%%%%%%%%%%%%%%%%%

We apply the oscillatory integral operator bounds of the previous section to the study of spectral clusters, as explained in the introduction. Our main result is the following.

\begin{theorem}\label{thm:main}
 Let $(M,g)$ a smooth compact Riemannian manifold of dimension $N\ge2$, without boundary. Denote by $\Delta_g$ the Laplace--Beltrami operator on $M$. For any $\lambda\ge1$, let $\Pi_\lambda=\1(\lambda^2\le\Delta_g<(\lambda+1)^2)$ the spectral projection of $\Delta_g$ onto the spectral cluster $E_\lambda=\Pi_\lambda L^2(M)$. Then, there exists $C>0$ such that for any orthonormal system $(f_j)_{j\in J}\subset E_\lambda$ and for any $(\nu_j)_{j\in J}\subset\C$, we have
 \begin{equation}\label{eq:est-main}
    \norm{\sum_{j\in J}\nu_j|f_j|^2}_{L^{p/2}(M)}\le C\lambda^{2s(p)}\left(\sum_{j\in J}|\nu_j|^{\alpha(p)}\right)^{1/\alpha(p)},
 \end{equation}
 where
 $$\left\{\begin{array}{lll}
    s(p)=N\left(\frac12-\frac1p\right)-\frac12, & \quad \alpha(p)=\frac{p(N-1)}{2N} & \quad \text{if}\ \ \frac{2(N+1)}{N-1}\le p\le+\ii,\\
    s(p)=\frac{N-1}{2}\left(\frac12-\frac1p\right), & \quad \alpha(p)=\frac{2p}{p+2} & \quad \text{if}\ \ 2\le p\le\frac{2(N+1)}{N-1}.
   \end{array}\right.
 $$
\end{theorem}

 As explained in \cite{FraSab-14}, bounds for systems of orthonormal functions like \eqref{eq:est-main} can be formulated in a more compact way using operators. Given an orthonormal system $(f_j)_{j\in J}\subset E_\lambda$ and coefficients $(\nu_j)_{j\in J}\subset\C$ as in the statement of Theorem \ref{thm:main}, one can build the operator 
 $$\gamma:=\sum_{j\in J}\nu_j|f_j\rangle\langle f_j|,$$
 where we used Dirac's notation $|u\rangle\langle v|$ for the operator $(|u\rangle\langle v|)g:=\langle v,g\rangle u$, for all $u,v,g\in L^2(M)$. Then, the quantity on the left side of \eqref{eq:est-main} is the \emph{density} associated to the operator $\gamma$:
 $$\sum_{j\in J}\nu_j|f_j(x)|^2=\gamma(x,x)=:\rho_\gamma(x),\ \forall x\in M,$$
 where $\gamma(\cdot,\cdot)$ denotes the integral kernel of $\gamma$. Hence, Theorem \ref{thm:main} can be reformulated as the following result.
 
 \begin{theorem}
  Under the same assumptions as in Theorem \ref{thm:main}, there exists $C>0$ such that for any operator $\gamma$ on $L^2(M)$ satisfying $\gamma\Pi_\lambda=\gamma=\Pi_\lambda\gamma$, we have
    \begin{equation}\label{eq:est-main-operators}
      \norm{\rho_\gamma}_{L^{p/2}(M)}\le C\lambda^{2s(p)}\norm{\gamma}_{\gS^{\alpha(p)}(L^2(M))},
    \end{equation}
  where the exponents $s(p)$ and $\alpha(p)$ are the same as in Theorem \ref{thm:main}, and the Schatten class $\gS^{\alpha(p)}$ was defined in Remark \ref{rk:schatten}.
 \end{theorem}

 \begin{remark}
  Again using \cite[Prop. 1]{FraSab-14}, the statement of Theorem \ref{thm:main} is equivalent to the estimate
  \begin{equation}\label{eq:schatten-Pilambda}
      \norm{W\Pi_\lambda\bar{W}}_{\gS^{\alpha(p)'}(L^2(M))}\le C\lambda^{2s(p)}\norm{W}_{L^{2p/(p-2)}(M)}^2,\ \forall W\in L^{2p/(p-2)}(M),
  \end{equation}
  for some $C>0$ independent of $W$ and $\lambda\ge 1$.
\end{remark}  

The next remark explains the connection between Theorem \ref{thm:main} and the results in our earlier paper \cite{FraSab-14} on $\R^N$. This discussion is analogous to the beginning of \cite[Ch. 5]{Sogge-book}.

\begin{remark}
In \cite[Thm. 2]{FraSab-14} we have proved a Schatten version of the Stein-Tomas inequality: namely for $\frac{2(N+1)}{N-1}\le p\le\ii$ we have the inequality 
  $$\norm{WT_{\Sph^{N-1}}\bar{W}}_{\gS^{\alpha(p)'}(L^2(\R^N))}\le C\norm{W}_{L^{2p/(p-2)}(M)}^2,\ \forall W\in L^{2p/(p-2)}(\R^N),$$
  where the operator $T_{\Sph^{N-1}}$ is the Fourier multiplier by a delta function on $\Sph^{N-1}$, 
  $$(T_{\Sph^{N-1}} f)(x)=\int_{\Sph^{N-1}}e^{ix\cdot\omega}\hat{f}(\omega)\,d\sigma(\omega),\,\forall x\in\R^N,\,\forall f\in L^1(\R^N),$$
and $d\sigma$ denotes the surface measure on $\Sph^{N-1}$. If we denote by $T^{(r)}$ the operator defined in an analogue fashion as $T_{\Sph^{N-1}}$ replacing the sphere of radius $1$ by the sphere of radius $r>0$, a scaling argument implies that 
  $$\norm{WT^{(r)}\bar{W}}_{\gS^{\alpha(p)'}(L^2(\R^N))}\le Cr^{2s(p)}\norm{W}_{L^{2p/(p-2)}(M)}^2,\ \forall W\in L^{2p/(p-2)}(\R^N).$$
In the case of $\R^N$, the operator $\Pi_\lambda$ is just the Fourier multiplier by the characteristic function of the annulus $\{\xi\in\R^N,\,\lambda^2\le|\xi|^2\le(\lambda+1)^2\}$ and hence, using 
  $$\Pi_\lambda=\int_{\lambda}^{\lambda+1}T^{(r)}\,dr$$
and the triangle inequality, we find that
  $$\norm{W\Pi_\lambda\bar{W}}_{\gS^{\alpha(p)'}(L^2(\R^N))}\le C\lambda^{2s(p)}\norm{W}_{L^{2p/(p-2)}(M)}^2,\ \forall W\in L^{2p/(p-2)}(\R^N).$$
One can then obtain the same inequality in the remaining range $2\le p\le \frac{2(N+1)}{N-1}$ by interpolating the $p=\frac{2(N+1)}{N-1}$ inequality (that we just obtained) with the trivial $p=2$ inequality (that just uses the fact that $\Pi_\lambda$ is a bounded operator on $L^2(\R^N)$ with operator norm $1$). Hence we have obtained the analogue of \eqref{eq:schatten-Pilambda} on $\R^N$ as a consequence of the Stein-Tomas inequality in Schatten spaces. Conversely, one may interpret \eqref{eq:schatten-Pilambda} as an averaged version of the Stein-Tomas inequality on a compact manifold $M$.
\end{remark}

\begin{remark}
From a wider perspective Theorem \ref{thm:main} (in particular, in the equivalent form \eqref{eq:schatten-Pilambda}) belongs to a class of trace ideal inequalities for operators of the form $\beta(\sqrt{\Delta_g})W$, where $W$ is a multiplication operator on $M$ and $\beta$ is a function on $[0,\infty)$. Such bounds have a long history on $\R^N$ (see \cite[Chp. 4]{Sim}) and the basic form of this inequality goes back to Kato, Seiler and Simon \cite[Thm. 4.1]{Sim}. This inequality has a simple generalization to manifolds which we record in the appendix (Theorem \ref{kss}) since we have not found it in the literature. While this inequality gives the optimal trace ideal for a large class of functions $\beta$, the point of Theorem \ref{thm:main} is that for $\beta(\tau) = \1(\lambda\le \tau\le\lambda+1)$ this general inequality can be improved by taking the oscillatory character of the eigenfunctions into account. This is done through the results from Section \ref{sec:carleson}.
\end{remark}

\begin{remark}\label{pqr}
We interpret the quotient $\norm{\rho_\gamma}_{L^{p/2}(M)}/\norm{\rho_\gamma}_{L^1(M)}$ as a measure of the concentration of the function $\rho_\gamma$. This can be made more quantitative using the bound
\begin{equation}
\label{eq:pqr}
\left|\left\{\rho_\gamma>\frac{\norm{\rho_\gamma}_{L^1}}{4\text{Vol}(M)}\right\}\right|\ge\frac12\left(\frac{p}{8}\right)^{\frac{2}{p-2}}\left(\frac{\norm{\rho_\gamma}_{L^1}}{\norm{\rho_\gamma}_{L^{p/2}}}\right)^{\frac{p}{p-2}} \quad\text{for}\ p>2.
\end{equation}
Thus, \eqref{eq:est-main-operators} shows that $\rho_\gamma$ concentrates on a set of measure at least $\lambda^{-\frac{2ps(p)}{p-2}}\left(\frac{\norm{\gamma}_{\gS^1}}{\norm{\gamma}_{\gS^{\alpha(p)}}}\right)^{\frac p{p-2}} $.

The bound \eqref{eq:pqr} follows for instance from the $pqr$-lemma \cite[Lem. 2.1]{FroLieLos-86} and we briefly sketch its proof. We have for any $0<\tau_1<\tau_2<\infty$
\begin{align*}
(\tau_2-\tau_1)|\{\rho_\gamma >\tau_1\}| & \geq \int_{\tau_1}^{\tau_2} |\{\rho_\gamma>\tau\}|\,d\tau = \norm{\rho_\gamma}_{L^1} - \int_{0}^{\tau_1} |\{\rho_\gamma>\tau\}|\,d\tau - \int_{\tau_2}^{\infty} |\{\rho_\gamma>\tau\}|\,d\tau \\
& \geq \norm{\rho_\gamma}_{L^1} - |M| \int_{0}^{\tau_1} \,d\tau - \tau_2^{-p/2+1} \int_0^{\infty} |\{\rho_\gamma>\tau\}|\tau^{p/2-1} \,d\tau \\
& \geq \norm{\rho_\gamma}_{L^1} - \tau_1 |M| - \tau_2^{-p/2+1} (2/p) \|f\|_{L^{p/2}}^{p/2} \,.
\end{align*}
We choose $\tau_1 = \|\rho_\gamma\|_{L^1}/(4|M|)$ and $\tau_2 = ((8/p)\|f\|_{L^{p/2}}^{p/2}/\|f\|_{L^1})^{2/(p-2)}$, so that the right side becomes $\|\rho_\gamma\|_{L^1}/2$, whereas the left side does not exceed $\tau_2|\{\rho_\gamma>\tau_1\}|$. This yields \eqref{eq:pqr}.
\end{remark}

We now discuss the optimality of \eqref{eq:est-main} and \eqref{eq:est-main-operators}. Since it involves two exponents $s(p)$ and $\alpha(p)$, optimality can be understood in several ways. We discuss here two basic notions of optimality and defer the discussion of a stronger notion to the next subsection.

\begin{remark}[Optimality of $s(p)$]\label{rk:optimality}
We claim that, whatever the value of $\alpha(p)$ is, the bound \eqref{eq:est-main} cannot hold with a smaller power of $\lambda$ than $2s(p)$. This follows from the optimality of Sogge's bound \eqref{eq:sogge-bound} by choosing only a single function (i.e., $|J|=1$), in which case the right side becomes independent of $\alpha(p)$.
\end{remark}

The optimality of the exponent $\alpha(p)$ is more delicate to discuss: indeed, since the space $E_\lambda$ is finite-dimensional, all the Schatten norms are equivalent on $E_\lambda$. However, estimating the Schatten norm in $\gS^\alpha$ by the norm in $\gS^{\beta}$ with $\beta>\alpha$ gives an additional factor $(\dim E_\lambda)^{1/\alpha-1/\beta}$, which grows with $\lambda$. Hence, one could artificially increase the exponent $\alpha(p)$ up to also increasing $s(p)$, but then the inequality would become non-optimal when $|J|=1$. As a consequence, the optimality of $\alpha(p)$ only makes sense when the power of $\lambda$ in \eqref{eq:est-main} is fixed, and we take it to be the sharp one $2s(p)$.

\begin{remark}[Optimality of $\alpha(p)$]\label{rk:optimality2}
We claim that, if $s(p)$ is given by the value in the theorem, then the bound \eqref{eq:est-main} cannot hold with a larger value of $\alpha(p)$ than that given in the theorem. This follows by taking $(f_j)$ to be an orthonormal basis of $E_\lambda$ and all $\nu_j=1$. Indeed, with this choice the left side of \eqref{eq:est-main} is bounded from below by
$$
\norm{\sum_{j\in J}|f_j|^2}_{L^{p/2}(M)} \ge |M|^{2/p-1} \norm{\sum_{j\in J}|f_j|^2}_{L^1(M)} = |M|^{2/p-1} (\dim E_\lambda)
$$
whereas the right side is given by $C \lambda^{s(p)} (\dim E_\lambda)^{1/\alpha(p)}$. Therefore it follows from \eqref{eq:lower-bound-dim-cluster} and the fact that
 $$2s(p)+\frac{N-1}{\alpha(p)}=N-1$$
that $\alpha(p)$ cannot be increased above the value given in the theorem (provided we insist on the value $s(p)$ from \eqref{eq:sogge-bound}).
\end{remark}

\begin{remark}
Just like Sogge's theorem, our Theorem \ref{thm:main} remains valid when $\Delta_g$ is replaced by a classical pseudo-differential operator of order 1 for which the cosphere $\{x\in T_x^*M:\ p(x,\xi)=1\}$ are strictly convex, where $p(x,\xi)$ denotes the principal symbol of the operator. This follows essentially by the same proof, except that we invoke the parametrix from \cite[Lemma 5.1.3]{Sogge-book}.
\end{remark}

 \begin{proof}[Proof of Theorem \ref{thm:main}]
 We prove the inequality \eqref{eq:schatten-Pilambda} for $p=2,\frac{2(N+1)}{N-1},\ii$ and then the general result follows by interpolation.

The case $p=2$ is trivial because $\alpha(p)'=\ii$ and $s(p)=0$, so it amounts to the bound $\norm{\Pi_\lambda}_{L^2\to L^2}\le C$, which is true with $C=1$ since $\Pi_\lambda$ is a projection.

For the case $p=\ii$, we use the fact that by the pointwise Weyl law \eqref{eq:weylbound} with $p=\infty$, there exists $C>0$ such that for all $x\in M$ one has
$$\Pi_\lambda(x,x)\le C\lambda^{N-1}.$$
By orthogonality, one has
 $$\int_M|\Pi_\lambda(x,y)|^2\,dy=\Pi_\lambda(x,x),\ \forall x\in M,$$
and thus 
 $$\norm{\Pi_\lambda(\cdot,\cdot)}_{L^\ii L^2(M\times M)}\le C\lambda^{\frac{N-1}{2}}.$$
This implies the trace class bound
 $$\norm{W\Pi_\lambda\bar{W}}_{\gS^1(L^2(M))} = \norm{W\Pi_\lambda}_{\gS^2(L^2(M))}^2=\int_M\int_M|W(x)|^2|\Pi_\lambda(x,y)|^2\,dx\,dy\le C\lambda^{N-1}\norm{W}_{L^2(M)}^2,$$
 the desired estimate for $p=\ii$.
 
We now come to the case $p=\frac{2(N+1)}{N-1}$, which is the core of the proof. The general strategy is the same as in the proof of \cite[Thm. 5.1.1]{Sogge-book}, which relies on a parametrix for the propagator $e^{it\sqrt{\Delta_g}}$. We could appeal directly to \cite[Lemma 5.1.3]{Sogge-book} (which goes back to H\"ormander \cite{Hormander-68} and is summarized, for instance, in \cite[Thm. 4]{BurGerTzv-07}), where such a parametrix is obtained even in the general case where $\sqrt{\Delta_g}$ is replaced by any classical pseudo-differential operator of order 1. However, we prefer to follow the strategy of the paper \cite{Sogge-89} based on the classical Hadamard parametrix construction, which is more elementary in the specific case of the Laplace--Beltrami operator. 

Let $\epsilon>0$ be a parameter that we will later choose small, but independent of $\lambda$. We claim that there is a Schwartz function $\chi$ on $\R$ satisfying $|\chi|^2>0$ on $[0,1]$ and $\supp\hat{\chi}\subset(0,\epsilon]$. In fact, let $\zeta\in C_0^\infty(0,\epsilon)$ with $\hat\zeta(0) = (2\pi)^{-1/2} \int_{\R}\zeta\,dt>0$. By continuity, $|\hat\zeta|^2>0$ on $[0,\Lambda]$ for some $\Lambda>0$, and then $\chi(\lambda):=\zeta(\min\{1,\Lambda\}\lambda)$ has the claimed properties.

We obtain the operator inequality
 \begin{equation}\label{ineq:smooth}
    0\le W\Pi_\lambda\bar{W}\le C W|\chi|^2(\sqrt{\Delta_g}-\lambda)\bar{W},
 \end{equation}
with $C=(\inf_{[0,1]}|\chi|^2)^{-1}$. Therefore, the claimed bound will follow if we can prove that for a suitable $\epsilon>0$
\begin{equation}
\label{eq:goal1}
\norm{W\chi(\sqrt{\Delta_g}-\lambda)}_{\gS^{2(N+1)}(L^2(M))}\le C\lambda^{\frac{N-1}{2(N+1)}}\norm{W}_{L^{N+1}(M)}.
\end{equation}

The operator $\chi(\sqrt{\Delta_g}-\lambda)$ can be described locally on $M$ using the following lemma, coming from \cite[Lem. 5.1.3]{Sogge-book} and stated in the version of \cite[Thm. 4]{BurGerTzv-07}.

\begin{lemma}\label{lem:parametrix}
 There exists $\epsilon>0$ such that for any Schwartz function $\chi$ on $\R$ with $\text{supp}\,\hat{\chi}\subset[-\epsilon,\epsilon]\setminus\{0\}$ and any $\lambda\ge1$, we have the decomposition
 $$\chi(\sqrt{\Delta_g}-\lambda)=K_\lambda+R_\lambda,$$
 where the integral kernel of the operator $R_\lambda$ satisfies
 \begin{equation}\label{eq:kernel-bound-remainder}
    \norm{R_\lambda(\cdot,\cdot)}_{L^\ii L^2(M\times M)}\le C
 \end{equation}
 for some $C>0$ independent of $\lambda$, and the operator $K_\lambda$ can be described locally in the following fashion. For any $x_0\in M$, there exist systems of coordinates $W\subset V\subset\R^N$ around $x_0$, and a function $a:V\times W\times\R_+\to C$ with the bounds
 \begin{equation}\label{eq:bounds-a}
    \forall\alpha,\beta\in\N^N,\,\exists C>0,\forall(x,y,\lambda)\in V\times W\times\R_+,\, |\partial^\alpha_x\partial^\beta_y a(x,y,\lambda)|\le C,
 \end{equation}
 which is furthermore supported on $\{(x,y,\lambda)\in V\times W\times\R_+,\, d_g(x,y)\sim\epsilon\}$, such that for all $(x,y,\lambda)\in V\times W\times\R_+$, 
 \begin{equation}\label{eq:kernel-fourier-integral}
    K_\lambda(x,y)=\lambda^{\frac{N-1}{2}}e^{-i\lambda d_g(x,y)}a(x,y,\lambda).
 \end{equation}
\end{lemma}

The proof of Lemma \ref{lem:parametrix} in \cite{Sogge-book} is valid not only for the Laplace--Beltrami operator but also for elliptic pseudo-differential operators. However, as pointed out in \cite{Sogge-89,Sogge-book-14}, it can be proved in a more elementary fashion (that is, not relying on pseudo-differential calculus) using the Hadamard parametrix. We recall briefly this construction in Appendix \ref{app:parametrix} for the sake of completeness.

The remainder term $R_\lambda$ is more regular than what we want to prove. In fact, \eqref{eq:kernel-bound-remainder} leads to the bound
$$\norm{WR_\lambda}_{\gS^2}\le\norm{W}_{L^2(M)}\norm{R_\lambda(\cdot,\cdot)}_{L^\ii L^2(M\times M)}\le C\norm{W}_{L^{N+1}(M)}.$$
On the other hand, after multiplying by a localizing partition of unity, we can consider $(\chi(\sqrt{\Delta_g}-\lambda)-R_\lambda)(x,y)$ as a function on $\R^N\times\R^N$ of the form \eqref{eq:kernel-fourier-integral}. The key observation now is that the phase function $(x,y)\mapsto d_g(x,y)$ satisfies the assumptions of Corollary \ref{coro} as explained, for instance, in \cite[p. 831--832]{DosKenSal-14}: the image of $y\mapsto (\nabla_x d_g)(x,y)$ is the geodesic sphere centered at $x$, which has non-zero curvature by Gauss' lemma. Therefore, Corollary \ref{coro} in the form \eqref{eq:coro-schatten} implies that
$$
\norm{\chi(\sqrt{\Delta_g}-\lambda)-R_\lambda}_{\gS^{2(N+1)}(L^2(\R^N))}\le C\lambda^{\frac{N-1}{2}-\frac{N(N-1)}{2(N+1)}}\norm{W}_{L^{N+1}(M)}=C\lambda^{\frac{N-1}{2(N+1)}}\norm{W}_{L^{N+1}(M)}.
$$
Here we implicitly sum over a finite partition of unity and estimate the Jacobian coming from the change of variables. This yields the claimed bound \eqref{eq:goal1}.
\end{proof}
 
%%%%%%%%%%%%%%%%%%%%%%%%%%%%%%%%%%%%%%%%%%
\section{Optimality}\label{sec:optimality}
%%%%%%%%%%%%%%%%%%%%%%%%%%%%%%%%%%%%%%%%%%

\subsection{Statement of the optimality result}

The goal of this section is to prove the optimality of the inequality
\begin{equation}\label{eq:density-bound-sphere}
  \norm{\rho_\gamma}_{L^{p/2}(\Sph^2)}\lesssim
  \begin{cases}
    \lambda^{\frac12-\frac1p}\norm{\gamma}_{\gS^{\frac{2p}{p+2}}}  & \text{if}\ 2\le p\le 6,\\
    \lambda^{1-\frac 4p}\norm{\gamma}_{\gS^{\frac{p}{4}}} & \text{if}\ 6\le p\le+\ii,
  \end{cases}
\end{equation}
for all $0\le\gamma\le\Pi_\lambda$, where $\Pi_\lambda$ is from \eqref{eq:pi} with $\Delta_g$ denoting the Laplace--Beltrami operator on $\Sph^2$ with respect to its standard metric. More precisely, we already mentioned that the power of $\lambda$ on the right side of \eqref{eq:density-bound-sphere} is optimal in two cases (and this is true on any $M$, not only for $M=\Sph^2$): (i) when $\text{rank}\,\gamma\sim1$, which was the case studied by Sogge and (ii) when $\text{rank}\,\gamma\sim\text{rank}\,\Pi_\lambda$ by Weyl's law. Hence, it is natural to look at the intermediate case where $1\ll\text{rank}\,\gamma\ll\text{rank}\,\Pi_\Lambda$. We prove that for any sequence $(r_\lambda)$ with 
$$
  \lim_{\lambda\to+\ii}r_\lambda=+\ii,\qquad \lim_{\lambda\to+\ii}\frac{r_\lambda}{\lambda}=0,
$$
there exists a sequence of projections $(\gamma_\lambda)$ such that $\text{rank}\,\gamma_\lambda\sim r_\lambda$ as $\lambda\to+\ii$ and 
$$
\norm{\rho_{\gamma_\lambda}}_{L^{p/2}(\Sph^2)}\gtrsim
  \begin{cases}
    \lambda^{\frac12-\frac1p}r_\lambda^{\frac12+\frac1p}  & \text{if}\ 2\le p\le 6,\\
    \lambda^{1-\frac 4p}r_\lambda^{\frac4p} & \text{if}\ 6\le p\le+\ii,
  \end{cases}
$$
for $\lambda$ large enough. Since $\text{rank}\,\Pi_\lambda\sim2\lambda$ in the case of $\Sph^2$, the assumptions on $r_\lambda$ are satisfied, for instance, for $r_\lambda=\lambda^{\zeta}$ with $\zeta\in(0,1)$.

Let us now explain how to build the projection $\gamma_\lambda$, by recalling what happens in the rank one case which was studied by Sogge. On $\Sph^2$, we denote the $L^2$-normalized spherical harmonics by $(Y^m_\ell)$, with $\ell\in\N$ and $-\ell\le m\le\ell$. They satisfy the equation
\begin{equation}\label{eq:eigenvalue-eq-spherical-harmonics}
  \Delta_{\Sph^2}Y^m_\ell=\ell(\ell+1)Y^m_\ell.
\end{equation}
The spherical harmonics $Y^{\pm\ell}_\ell$ saturate the Sogge bound \eqref{eq:sogge-bound} in the range $2\le p\le 6$, in the sense that
$$\norm{|Y^{\pm\ell}_\ell|^2}_{L^{p/2}(\Sph^2)}\gtrsim\ell^{\frac12-\frac1p},$$
for $\ell$ large enough. This can be seen by explicit computation, using the fact that
$$Y^{\pm\ell}_\ell(\theta,\phi)=c_\ell(\sin\theta)^\ell e^{\pm i\ell\phi},\ \forall(\theta,\phi)\in[0,\pi]\times[0,2\pi],$$
where $c_\ell$ is a normalization constant. On the other hand, the spherical harmonics $Y^0_\ell$ saturate Sogge's bound \eqref{eq:sogge-bound} in the range $6\le p\le+\ii$:
$$\norm{|Y^0_\ell|^2}_{L^{p/2}(\Sph^2)}\gtrsim\ell^{1-\frac4p}.$$
These two facts in hand, it is not a surprise that the saturation of \eqref{eq:density-bound-sphere} happens when considering several $Y^m_\ell$, with $m\sim\ell$ in the case $2\le p\le 6$ and $m\ll\ell$ in the case $6\le p\le+\ii$.

\begin{theorem}\label{thm:optimality}
 Let $(r_\ell)_{\ell\in\N}\subset\R_+$ a sequence such that
 \begin{equation}
  \lim_{\ell\to+\ii}r_\ell=+\ii,\qquad \lim_{\ell\to+\ii}\frac{r_\ell}{\ell}=0.
 \end{equation}
 For $\ell$ large enough such that $r_\ell\le\ell/2$, define
\begin{align*}
\gamma_\ell^{(2)} & :=
     \displaystyle \sum_{\ell-2r_\ell< m\le \ell-r_\ell}|Y^m_\ell\rangle\langle Y^m_\ell| \\
\gamma_\ell^{(\ii)} & :=
     \displaystyle \sum_{r_\ell\le m< 2r_\ell}|Y^m_\ell\rangle\langle Y^m_\ell|
\end{align*}
Then there are $c>0$ and $L\ge1$ such that for all $\ell\ge L$ and $2\le p\le\infty$ we have
\begin{align*}
\norm{\rho_{\gamma_\ell^{(2)}}}_{L^{p/2}(\Sph^2)}\ge
    c\,\ell^{\frac12-\frac1p}r_\ell^{\frac12+\frac1p} \,, \\
    \norm{\rho_{\gamma_\ell^{(\ii)}}}_{L^{p/2}(\Sph^2)}\ge
    c\,\ell^{1-\frac4p}r_\ell^{\frac4p} \,.
\end{align*}
In particular, the bound \eqref{eq:density-bound-sphere} is saturated by $\gamma_\ell^{(2)}$ for $2\le p\le 6$ and by $\gamma_\ell^{(\ii)}$ for $6\le p\le\infty$.
\end{theorem}

\begin{remark}
 The notation $\gamma_\ell^{(2)}$ is motivated by the fact that this operator saturates the inequality for $p$ close to $2$, while $\gamma_\ell^{(\ii)}$ saturates the inequality for $p$ close to $\ii$.
\end{remark}

We prove the following pointwise bounds on $\rho_{\gamma_\ell^{(\#)}}$. We parametrize points in $\Sph^2$ as usual by $(\theta,\phi)\in(0,\pi)\times(0,2\pi)$, which stands for the point $(\sin\theta\cos\phi,\sin\theta\sin\phi,\cos\theta)\in\R^3$.

\begin{proposition}\label{prop:WKB}
 Let $(r_\ell)$ and $(\gamma_\ell^{(\#)})$ be as in Theorem \ref{thm:optimality}. 
 \begin{enumerate}
  \item There are $c>0$, $\eta_2>0$, and $L\ge1$ such that for all $\ell\ge L$, for all $0\le\theta\le\eta_2(r_\ell/\ell)^{1/2}$, and for all $\phi\in[0,2\pi]$ we have
  \begin{equation}
    \rho_{\gamma_\ell^{(2)}}(\pi/2-\theta,\phi)\ge c(\ell r_\ell)^{1/2}.
  \end{equation}
  
  \item There are $c>0$, $\eta_1>0$, and $L\ge1$ such that for all $\ell\ge L$, for all $\eta_1 r_\ell/\ell\le\theta\le \pi/2$, and for all $\phi\in[0,2\pi]$ we have
  \begin{equation}
    \rho_{\gamma_\ell^{(\ii)}}(\theta,\phi)\ge c \ \frac{r_\ell}{\sin\theta}.
  \end{equation}
 \end{enumerate}
\end{proposition}

In fact, our proof will show that, with possibly different constants, the reverse inequalities in the proposition hold as well in the same parameter regime. In Subsection \ref{sec:heuristics} we provide a heuristic semi-classical interpretation of this proposition.

In the case $\#=2$, this proposition implies a concentration of $\rho_{\gamma_\ell^{(2)}}$ on a neighborhood of the equator of area $(r_\ell/\ell)^{1/2}$, with an amplitude $(\ell r_\ell)^{1/2}$. This is coherent with the case $r_\ell\sim1$, knowing that $|Y^\ell_\ell|^2$ concentrates on a neighborhood of size $\ell^{-1/2}$ of the equator, with an amplitude $\ell^{1/2}$. In the case $\#=\ii$, the proposition implies a concentration of $\rho_{\gamma_\ell^{(\ii)}}$ on a neighborhood of the north pole of area $(r_\ell/\ell)^2$ (recall that the area measure on the sphere is $\sin\theta\,d\theta\,d\phi$), with an amplitude $\ell$. This is coherent with the case $r_\ell\sim1$, knowing that $|Y^0_\ell|^2$ concentrates on a neighborhood of the poles of area $1/\ell^2$, with an amplitude $\ell$ (see \cite[Lem. 2.1]{Sogge-86}).

In passing we we mention that the above region of concentration contains the $L^1$-norm of order $r_\ell = \Tr\gamma_\ell^{(\#)}$ for case $\#=2$, whereas the $L^1$ norm coming from the region of concentration is only $o(r_\ell)$ for $\#=\ii$.

\begin{proof}[Proof of Theorem \ref{thm:optimality} assuming Proposition \ref{prop:WKB}]
By the first estimate of Proposition \ref{prop:WKB} we have for all $\ell\ge L$, 
 \begin{align*}
    \norm{\rho_{\gamma_\ell^{(2)}}}_{L^{p/2}(\Sph^2)}^{p/2} &\ge \int_0^{2\pi}\int_0^{\eta_2(r_\ell/\ell)^{1/2}}\rho_{\gamma_\ell^{(2)}}(\pi/2-\theta,\phi)^{p/2}\cos\theta \,d\theta\,d\phi\\
    &\ge\pi C^{p/2}\eta_2\ell^{\frac p4-\frac12}r_\ell^{\frac p4+\frac12},
 \end{align*}
 which has the desired behaviour in $\ell$ (we used the fact that $\cos\theta\ge1/2$ for all $\theta\in[0,\eta_2(r_\ell/\ell)^{1/2})$, for $\ell$ large enough).
 
 Similarly, by the second estimate of Proposition \ref{prop:WKB} we have for all $\ell\ge L$,
  \begin{align*}
   \norm{\rho_{\gamma_\ell}}_{L^{p/2}(\Sph^2)}^{p/2} &\ge \int_0^{2\pi}\int_{\eta_1r_\ell/\ell}^{2\eta_1 r_\ell/\ell}\rho_{\gamma_\ell^{(\ii)}}(\theta,\phi)^{p/2}\sin\theta\,d\theta\,d\phi\\
   &\ge \pi C^{p/2} r_\ell^{\frac p2}\int_{\eta_1r_\ell/\ell}^{2\eta_1r_\ell/\ell} \sin^{-\frac p2 + 1}\theta \,d\theta\ge \pi C^{p/2} \frac{1-2^{-\frac{p}{2}+2}}{\frac p2 -2}\, \eta_1^{2-\frac p2}\, \ell^{\frac p2-2}\, r_\ell^2,
  \end{align*}
(with $(1-2^{-p/2+2})/(p/2-2)$ interpreted as $\ln2$ if $p=4$), which has the right behaviour in $\ell$ (we used the fact that $\sin\theta\le\theta$ for all $\theta$). 
\end{proof}

The rest of this section will be devoted to the proof of Proposition \ref{prop:WKB}. It follows from WKB bounds on the spherical harmonics $Y^m_\ell$, in the two regimes $m\sim r_\ell$ or $\ell-m\sim r_\ell$. Asymptotics of single spherical harmonics in these regimes are already known (see for instance \cite{Olver-75}), but for the sake of completeness we present in the appendix a proof of the estimates that we need. Once the behaviour of a \emph{single} $Y^m_\ell$ is understood, one has to \emph{sum} the $|Y^m_\ell|^2$ to obtain $\rho_{\gamma_\ell^{(\#)}}$. This is a serious difficulty which we have not seen discussed in the literature. The problem are the oscillations in $Y^m_\ell$ and our key to controlling them is Lemma \ref{lem:phases}, where the numbers $\eta_1$ and $\eta_2$ will be determined.

\subsection{Ingredients in the proof of the optimality result}

It is well-known that the $Y_\ell^m$ are of the form
\begin{equation}
\label{eq:ylmfactor}
Y^m_\ell(\theta,\phi)=e^{im\phi}g^m_\ell(\theta),
\end{equation}
and therefore our task is to find lower bounds on the functions $g^m_\ell$. (The functions $g_\ell^m$ are associated Legendre polynomials and we recall some facts about these functions in the appendix.) It is somewhat more convenient for us to work instead with the functions 
\begin{equation}
\label{eq:vlmdef}
v^m_\ell(\theta):=(\cos\theta)^{1/2}g^m_\ell\left(\frac{\pi}{2}-\theta\right),\qquad -\frac{\pi}{2}\le\theta\le\frac{\pi}{2} \,.
\end{equation}
As we will explain in the appendix, the functions $v^m_\ell$ satisfy the equations
\begin{equation}
\label{eq:vlmeq}
-\frac{d^2}{d\theta^2}v^m_\ell+Q_{\ell,m}v^m_\ell=0
\qquad\text{on}\ (-\pi/2,\pi/2)
\end{equation}
with the normalizations
\begin{equation}
\label{eq:vlmnorm}
\int_{-\pi/2}^{\pi/2}|v^m_\ell(\theta)|^2\,d\theta=\frac{1}{2\pi} \,.
\end{equation}
Here we have set
$$Q_{\ell,m}(\theta):=\frac{m^2-\frac14}{\cos^2\theta}-\frac14-\ell(\ell+1) \,.$$

We will approximate $v_\ell^m$ by the WKB method on the following interval $I_\ell^{(\#)}$, depending on the case $\#=2$ or $\ii$,
$$
I_\ell^{(\#)}:=
\begin{cases}
 \left(-\eta_2(r_\ell/\ell)^{1/2},\eta_2(r_\ell/\ell)^{1/2}\right) & \text{if}\ \#=2,\\
 \left(-\pi/2+\eta_1 r_\ell/\ell,\pi/2-\eta_1 r_\ell/\ell\right) & \text{if}\ \#=\ii \,.
\end{cases}
$$
Here $\eta_1$ and $\eta_2$ are two parameters to be determined later on. Before introducing the WKB approximations, we collect some bounds on the size of $Q_{\ell,m}$ on the interval $I_\ell^{(\#)}$. The proof is postponed to the appendix.

\begin{lemma}\label{lem:parameters-estimates}
Let $\eta_1>2$ and $\eta_2 <\sqrt 2$.
\begin{enumerate}
 \item There are $c_1,c_2>0$ and $L\ge1$ such that for all $\ell\ge L$, all $\ell-2r_\ell\le m\le\ell-r_\ell$, and all $\theta\in I_\ell^{(2)}$, 
 \begin{equation}
  -c_1\ell r_\ell\le Q_{\ell,m}(\theta) \le - c_2\ell r_\ell.
 \end{equation}
 \item There are $c_1,c_2>0$ and $L\ge1$ such that for all $\ell\ge L$, all $r_\ell\le m\le2r_\ell$, and all $\theta\in I_\ell^{(\ii)}$,
 \begin{equation}
  -c_1\ell^2 \le Q_{\ell,m}(\theta) \le -c_2 \ell^2.
 \end{equation}
\end{enumerate}
\end{lemma}

This lemma implies, in particular, that $Q_{\ell,m}< 0$ on $I_\ell^{(\#)}$ for all sufficiently large $\ell$.

We now define the WKB approximations
\begin{align*}
y_{\ell,m}:=
\begin{cases}
\displaystyle \frac{\cos(S_{\ell,m})}{|Q_{\ell,m}|^{1/4}} & \text{if}\ \ell+m\,\text{even},\\
\displaystyle \frac{\sin(S_{\ell,m})}{|Q_{\ell,m}|^{1/4}} & \text{if}\ \ell+m\,\text{odd},
\end{cases}
\end{align*}
where
$$S_{\ell,m}(\theta):=\int_0^\theta\sqrt{|Q_{\ell,m}(t)|}\,dt,\quad \forall\theta\in (-\pi/2,\pi/2).$$
The following proposition states that (a constant multiple of) the $y_{\ell,m}$ is a good approximation to $v_\ell^m$. The proof is more or less standard, but we present it for the sake of completeness in the appendix.

\begin{proposition}[WKB approximation]\label{wkbsingle}
Let $\eta_1>2$ and $\eta_2<\sqrt 2$. There is are $C>0$, $L\ge 1$ and $c_{\ell,m}$ such that for $\ell\ge L$ on $I_\ell^{(\#)}$,
\begin{equation}\label{eq:wkbsingle}
|v^m_\ell-c_{\ell,m} y_{\ell,m}|\le C r_\ell^{-1}|c_{\ell,m}| |Q_{\ell,m}|^{-1/4} \,.
\end{equation}
\end{proposition}

Note that $y_{\ell,m}$ is (at least on average) comparable with $|Q_{\ell,m}|$, and therefore the remainder in \eqref{eq:wkbsingle} is by a factor of $r_\ell^{-1}$ smaller than the main term.

The next lemma discusses the behavior of the normalization constants as $\ell\to\infty$. While one can probably give a self-contained proof of this result, we will use explicit formulas of spherical harmonics and Legendre functions. We defer the proof to the appendix.

\begin{lemma}\label{lem:normalization-constant}
Let $\eta_1>2$ and $\eta_2<\sqrt 2$. There are $C>0$ and $L\ge1$ such that for all $\ell\ge L$ and for all $\ell-2r_\ell< m\le\ell-r_\ell$ or $r_\ell\le m<2r_\ell$ we have
 $$C\ell\le|c_{\ell,m}|^2\le C^{-1}\ell.$$
\end{lemma}

We now come to the the crucial ingredient in our proof of Proposition \ref{prop:WKB}, namely a lower bound on the WKB approximations. Its proof will be given in the next subsection.

\begin{proposition}[Control of the oscillations]\label{wkbsum}
There are $\eta_1,\eta_2,c>0$ and $L$ such that for $\ell\ge L$ on $I_\ell^{(\#)}$,
\begin{equation}\label{eq:second-last-estimate}
  \sum_m|c_{\ell,m} y_{\ell,m}(\theta)|^2\ge c \sum_m |c_{\ell,m}|^2 |Q_{\ell,m}(\theta)|^{-1/2} \,,
\end{equation}
where the sum is over all $\ell-2r_\ell< m\le\ell-r_\ell$ or $r_\ell\le m<2r_\ell$.
\end{proposition}

With these ingredients at our disposal we will now complete the

\begin{proof}[Proof of Proposition \ref{prop:WKB} assuming Proposition \ref{wkbsum}]
According to \eqref{eq:ylmfactor} and \eqref{eq:vlmdef} we have for all $\theta\in (0,\pi)$
\begin{equation}\label{eq:comparison-Y-v}
  \sum_{m}|Y^m_\ell(\theta,\phi)|^2=\frac{1}{\sin\theta}\sum_{m}|v^{m}_\ell(\pi/2-\theta)|^2=\frac{1}{\cos(\pi/2-\theta)}\sum_{m}|v^{m}_\ell(\pi/2-\theta)|^2,
\end{equation}
where the sum is taken over $\ell-2r_\ell< m\le \ell-r_\ell$ if $\#=2$ and $r_\ell\le m< 2r_\ell$ if $\#=\ii$.

In order to bound the right side from below, we first observe that if $a,b\in\C$ satisfy $|a-b|\le \epsilon c$ and $|b|\le c$, then
\begin{equation}
\label{eq:elementary}
|a|^2 \ge \frac{1-\epsilon/2}{1+\epsilon/2} |b|^2 - \frac{\epsilon}{1+\epsilon/2} c^2 \,.
\end{equation}
In fact,
\begin{align*}
|a|^2 - |b|^2 & =(|a|-|b|) (|a|+|b|) \ge -|a-b|(|a|+|b|) \ge - \epsilon c(|a|+|b|) \\
& \ge -\epsilon\left( \frac12 |a|^2 + \frac12|b|^2 + c^2 \right).
\end{align*}

We apply \eqref{eq:elementary} with $a=v_\ell^m(\theta)$, $b=c_{\ell,m} y_{\ell,m}(\theta)$, $c=|c_{\ell,m}| |Q_{\ell,m}|^{-1/4}$ and $\epsilon=Cr_\ell^{-1}$ from \eqref{eq:wkbsingle}. Since $r_\ell\to \infty$, we obtain for all sufficiently large $\ell$ that on $I_\ell^{(\#)}$
$$
|v_\ell^m(\theta)|^2 \ge c \left( |c_{\ell,m} y_{\ell,m}(\theta)|^2 - C r_\ell^{-1} |c_{\ell,m}|^2 |Q_{\ell,m}|^{-1/2} \right) .
$$
Summing over the same range of $m$ as in \eqref{eq:comparison-Y-v} and applying Proposition \ref{wkbsum} we obtain
$$
\sum_m |v_\ell^m(\theta)|^2 \ge c' \left(1 - C' r_\ell^{-1} \right) \sum_m |c_{\ell,m}|^2 |Q_{\ell,m}(\theta)|^{-1/2} \,.
$$
Finally, the bounds from Lemmas \ref{lem:parameters-estimates} and \ref{lem:normalization-constant} yield for all sufficiently large $\ell$,
$$
\sum_m |v_\ell^m(\theta)|^2 \ge \begin{cases}
  c'(\ell r_\ell)^{1/2} & \text{if}\ \#=2,\\
  c'r_\ell & \text{if}\ \#=\ii.
 \end{cases}
$$
In view of \eqref{eq:comparison-Y-v} this is the claimed bound for $\#=\ii$. For $\#=2$ we use, in addition, $\cos\theta\le 1$ to get the claimed bound.
\end{proof}

%%%%%%%%%%%%%%%%%%%%%%%%%%%%%%%%%%%%%%

\subsection{Proof of Proposition \ref{wkbsum}}

The following lemma controls the oscillations of the WKB approximation and is the technical key step in the proof of our optimality result.

\begin{lemma}\label{lem:phases}
There are $A>0$, $\eta_1>0$, $\eta_2>0$ and $L\ge1$ such that for all $\ell\ge L$ and all $\theta\in I_\ell^{(\#)}$ we have
$$
\left| \sum_m e^{i(2S_{\ell,m}(\theta) +m\pi)} \right|\le A \,,
$$
where the sum is over $\ell-2r_\ell\le m \le \ell -r_\ell$ if $\#=2$ and $r_\ell\le m\le 2r_\ell$ if $\#=\ii$.
\end{lemma}

Before proving this lemma, we use it to deduce Proposition \ref{wkbsum}.

\begin{proof}[Proof of Proposition \ref{wkbsum}]
Because of Lemmas \ref{lem:normalization-constant} and \ref{lem:parameters-estimates} (with $q_\ell = \ell r_\ell$ if $\#=2$ and $q_\ell = \ell^2$ if $\#=\ii$) we have
\begin{align*}
\sum_m |c_{\ell,m} y_{\ell,m}|^2 & \ge c\ell \sum_m |y_{\ell,m}|^2 \\
& \ge c' \ell q_\ell^{-1/2} \sum_m |Q_{\ell,m}|^{1/2} |y_{\ell,m}|^2 \\
& = \frac{c'\ell}{2\sqrt{q_\ell}} \sum_m \left( 1+(-1)^{\ell+m}\cos(2S_{\ell,m}) \right) \\
& = \frac{c'\ell}{2\sqrt{q_\ell}} \left( r_\ell +(-1)^\ell \re \sum_m e^{i(2S_{\ell,m}+m\pi)} \right).
\end{align*}
According to Lemma \ref{lem:phases} we finally conclude
$$
\sum_m |c_{\ell,m} y_{\ell,m}|^2 \ge \frac{c'\ell}{2\sqrt{q_\ell}} r_\ell \left( 1 - A r_\ell^{-1} \right).
$$
On the other hand, again by Lemmas \ref{lem:normalization-constant} and \ref{lem:parameters-estimates},
\begin{align*}
\sum_m |c_{\ell,m}|^2 |Q_{\ell,m}|^{-1/2} & \le C\ell \sum_m |Q_{\ell,m}|^{-1/2} \\
& \le C' \ell q_\ell^{-1/2} r_\ell \,,
\end{align*}
which proves the result for large enough $\ell$.
\end{proof}

Thus, it remains to prove Lemma \ref{lem:phases}. Our main tool is the following Kuzmin--Landau inequality \cite[Thm 2.1]{GraKol-book} (see also \cite{Mordell-58}).

\begin{lemma}\label{kuzmanlandau}
Let $(\Phi_k)_{k=0}^K$ be numbers such that, for some $0<\epsilon\le\pi$,
$$
\epsilon\le \Phi_K-\Phi_{K-1} \le \Phi_{K-1}-\Phi_{K-2} \le \ldots\le \Phi_1-\Phi_0 \le 2\pi-\epsilon \,.
$$
Then
$$
\left| \sum_{k=0}^K e^{i\Phi_k} \right| \le \cot\frac{\epsilon}{4} \,.
$$
\end{lemma}

We recall the proof of this inequality for the sake of completeness. 

\begin{proof}
Let $\cS:=\sum_{k=1}^K e^{i\Phi_k}$ and $h_k:=\Phi_k-\Phi_{k-1}$ for $1\le k \le K$. The elementary equality
 $$e^{i\Phi_k}=\frac{e^{i\Phi_k}-e^{i\Phi_{k-1}}}{2i}\cot\left(\frac{h_k}{2}\right)+\frac{e^{i\Phi_k}-e^{i\Phi_{k-1}}}{2} \,,\qquad 1\le k\le K\,,$$
and summation by parts (Abel transform) imply that 
 \begin{align*}
    \cS=\frac{1}{2i} \sum_{k=1}^{K-1}e^{i\Phi_k}\left[\cot\left(\frac{h_k}{2}\right)-\cot\left(\frac{h_{k+1}}{2}\right)\right] & - \frac{1}{2i} e^{i\Phi_0} \left(i + \cot\left(\frac{h_1}{2} \right)\right) \\
   & +\frac1{2i} e^{i\Phi_K} \left(i+\cot\left(\frac{h_{K}}{2}\right) \right).
 \end{align*}
 As a consequence, we may estimate
 $$|\cS|\le\frac12\left(\sum_{k=1}^{K-1}\left|\cot\left(\frac{h_k}{2}\right)-\cot\left(\frac{h_{k+1}}{2}\right)\right|+ \frac{1}{\sin(h_1/2)} + \frac{1}{\sin(h_{K}/2)} \right).$$
Since $\cot$ in decreasing in $(0,\pi)$ and $h_k$ is non-increasing, we have
 \begin{align*}
    \sum_{k=1}^{K-1}\left|\cot\left(\frac{h_k}{2}\right)-\cot\left(\frac{h_{k+1}}{2}\right)\right| 
    &=\sum_{k=1}^{K-1} \left( \cot\left(\frac{h_{k+1}}{2}\right) - \cot\left(\frac{h_k}{2}\right) \right)\\
    &=\cot\left(\frac{h_K}{2}\right)- \cot\left(\frac{h_1}{2}\right).
 \end{align*}
Finally, we notice that
$$
\frac{1}{\sin(h_1/2)} - \cot\left(\frac{h_1}{2}\right) = \tan\left(\frac{h_1}{4}\right),
\qquad
\frac{1}{\sin(h_{K}/2)} + \cot\left(\frac{h_K}{2}\right) = \cot\left(\frac{h_K}{4}\right),
$$ 
and bound $\tan(h_1/4)\le \tan((2\pi-\epsilon)/4) = \cot(\epsilon/4)$ and $\cot(h_K/4)\le \cot(\epsilon/4)$.
\end{proof}

Finally, we are in position to give the

\begin{proof}[Proof of Lemma \ref{lem:phases}]
Since $S_{\ell,m}$ is an odd function, we only need to prove the inequality on $I_\ell^{(\#)}\cap[0,\infty)$. We want to deduce it from the Kuzmin--Landau inequality with $\Phi_m = 2S_{\ell,m} + \pi m$ and therefore we want to prove that
$$
\Phi_m-\Phi_{m-1} = 2\left( S_{\ell,m} - S_{\ell,m-1} \right) + \pi
$$
is non-increasing and separated away from $0$ and $2\pi$.

In order to prove monotonicity we use concavity of the square root and compute on $I_\ell^{(\#)}$ (where $Q_{\ell,m}\le 0$)
\begin{align*}
\frac12 \left( \sqrt{|Q_{\ell,m+1}|} + \sqrt{|Q_{\ell,m-1}|} \right)
& \le \sqrt{\frac12 \left( |Q_{\ell,m+1}| + |Q_{\ell,m-1}| \right)} \\
& = \sqrt{ |Q_{\ell,m}| - \frac{1}{\cos^2\theta}} \\
& \le \sqrt{ |Q_{\ell,m}| } \,.
\end{align*}
Integrating this inequality, we obtain $S_{\ell,m+1}-S_{\ell,m}\le S_{\ell,m}- S_{\ell,m-1}$ on $I_\ell^{(\#)}\cap[0,\infty)$, which implies $\Phi_{m+1}-\Phi_m \le \Phi_m-\Phi_{m-1}$ on this set. This is the claimed monotonicity.

In order to prove that $\Phi_m-\Phi_{m-1}$ is separated away from zero and $2\pi$ we write
$$
\sqrt{|Q_{\ell,m-1}|} - \sqrt{|Q_{\ell,m}|} = \frac{|Q_{\ell,m-1}|-|Q_{\ell,m}|}{\sqrt{|Q_{\ell,m-1}|} + \sqrt{|Q_{\ell,m}|}}
$$
and compute
$$
|Q_{\ell,m-1}| - |Q_{\ell,m}| = \frac{2m-1}{\cos^2\theta} \,.
$$
Thus, on $[0,\pi/2)$
$$
0 \le S_{\ell,m-1} - S_{\ell,m} = (2m-1) \int_0^\theta \frac{dt}{\cos^2t \ \left(\sqrt{|Q_{\ell,m-1}(t)|} + \sqrt{|Q_{\ell,m}(t)|} \right)} \,.
$$
We now bound $Q_{\ell,m}$ by Lemma \ref{lem:parameters-estimates}, use $\theta\in I_\ell^{(\#)}$ and recall the bounds on $m$ in the respective cases. For $\#=2$ we get
$$
0 \le S_{\ell,m-1} - S_{\ell,m} \le C \ell (\ell r_\ell)^{-1/2} \int_0^{\eta_2 (r_\ell/\ell)^{1/2}} \frac{dt}{\cos^2 t} \le C' \ell (\ell r_\ell)^{-1/2} \eta_2 (r_\ell/\ell)^{1/2} = C' \eta_2 \,.
$$
This implies
$$
\pi \ge \Phi_m - \Phi_{m-1} \ge - 2C'\eta_2 + \pi \,.
$$
This is the required separation condition if we choose $\eta_2<\pi/(2C')$. (We emphasize that this argument is not circular. The constant $C'$ seems to depend on $\eta_2$ through the use of Lemma \ref{lem:parameters-estimates}, but, in fact, we can make this constant independent of $\eta_2$ by applying the lemma with some fixed $\eta_2<\sqrt2$. Therefore the constant is not affected if afterwards we decrease $\eta_2$.)

For $\#=\ii$ we get similarly
$$
0 \le S_{\ell,m-1} - S_{\ell,m} \le C r_\ell (\ell^2)^{-1/2} \int_0^{\pi/2 - \eta_1 r_\ell/\ell} \frac{dt}{\cos^2 t} \le C' r_\ell (\ell^2)^{-1/2} (\ell/(\eta_1 r_\ell)) = C' \eta_1^{-1} \,.
$$
This implies
$$
\pi \ge \Phi_m - \Phi_{m-1} \ge - 2C' \eta_1^{-1} + \pi
$$
and the required separation follows if we choose $\eta_1>2C'/\pi$. (Similarly as before, enlarging $\eta_1$ is compatible with Lemma \ref{lem:parameters-estimates}.)

The lemma now follows from Lemma \ref{kuzmanlandau}.
\end{proof}
 
%%%%%%%%%%%%%%%%%%%%%%%%%%%%%%%%%%%%%%

\subsection{Heuristics}\label{sec:heuristics}

Finally, we would like to provide some heuristic explanation of the pointwise bounds in Proposition \ref{prop:WKB}. The arguments in this subsection are rather informal and it is not clear to us how to make them rigorous, which is why we have chosen an alternative approach. Nevertheless we think they provide an intuitive picture which might be useful in related situations.

Multiplying the equation \eqref{eq:glmeq} for $g_\ell^m$ by $\sin^2\theta$, we obtain 
$$-\sin\theta\frac{d}{d\theta}\sin\theta\frac{d}{d\theta}g^m_\ell-\ell(\ell+1)\sin^2\theta g^m_\ell=-m^2g^m_\ell,$$
which we may rewrite as
\begin{equation}\label{eq:semi-classics-gmell}
  -\frac{1}{\sin\theta}\frac{d}{d\theta}\sin\theta\sin^2\theta\frac{d}{d\theta}g^m_\ell+\sin2\theta\frac{d}{d\theta}g^m_\ell-\ell(\ell+1)\sin^2\theta g^m_\ell=-m^2g^m_\ell.
\end{equation}
The operator
$$H=-(\ell(\ell+1))^{-1}\frac{1}{\sin\theta}\frac{d}{d\theta}\sin\theta\sin^2\theta\frac{d}{d\theta}-\sin^2\theta$$
has the semi-classical form
$$H=-h^2\frac{1}{w}\frac{d}{d\theta}wa\frac{d}{d\theta}+V$$
with
$$h=(\ell(\ell+1))^{-1/2},\quad w(\theta)=\sin\theta,\quad a(\theta)=\sin^2\theta,\quad V(\theta)=-\sin^2\theta.$$
The operator $H$ is self-adjoint on $L^2((0,\pi),w(\theta)\,d\theta)$, and semi-classically one expects to have the pointwise asymptotics for all $\theta\in(0,\pi)$
\begin{align*}
  \1(H\in(A,B))(\theta,\theta) &\sim_{h\to0}\frac{1}{2\pi}\frac{1}{w(\theta)}|\{\xi\in\R,\,h^2a(\theta)\xi^2+V(\theta)\in(A,B)\}|\\
  &=\frac{1}{\pi h}\frac{1}{w(\theta)}\left(\left(\frac{B-V(\theta)}{a(\theta)}\right)^{1/2}_+-\left(\frac{A-V(\theta)}{a(\theta)}\right)^{1/2}_+\right).
\end{align*}
(To check that this is, indeed, the right scaling for semi-classics, one can verify it when $w$, $a$, and $V$ are constant functions). The functions $g^m_\ell$ are not exactly the eigenfunctions of the operator $H$, due to the additional (non self-adjoint) term $h^2\sin2\theta \ d/d\theta$ in the equation \eqref{eq:semi-classics-gmell} for $g^m_\ell$. However, since $hd/d\theta$ is semi-classically of order one, the term $h^2\sin2\theta d/d\theta$ should formally be of lower order and not contribute to the asymptotics. As a consequence, we should have
$$\sum_{a\le m\le b}|g^m_\ell(\theta)|^2\sim_{\ell\to\ii}\frac{\ell}{\pi\sin\theta}\left(\left(\frac{\sin^2\theta-b^2/(\ell(\ell+1))}{\sin^2\theta}\right)^{1/2}_+-\left(\frac{\sin^2\theta-a^2/(\ell(\ell+1))}{\sin^2\theta}\right)^{1/2}_+\right).$$
In the regime $\#=\ii$, we have $a=r_\ell$, $b=2r_\ell$, $\theta\sim r_\ell/\ell$, and in the regime $\#=2$ we have $a=\ell-2r_\ell$, $b=\ell-r_\ell$, $\theta\sim\pi/2-(r_\ell/\ell)^{1/2}$, which give the same asymptotics as in Proposition \ref{prop:WKB}. This concludes our heuristic derivation of Proposition \ref{prop:WKB}.

%%%%%%%%%%%%%%%%%%%%%%%%%%%%%%%%%%%%%%%%%%%%%%%%%%%%%%

\appendix

\section{The wave propagator in local coordinates}\label{app:parametrix}

In this appendix we sketch the proof of Lemma \ref{lem:parametrix}, following the arguments of \cite{Sogge-89,Sogge-book-14}. We first write 
$$
\chi(\sqrt{\Delta_g}-\lambda)= A_\lambda - \chi(-\sqrt{\Delta_g}-\lambda)
$$
and we will make use of the Fourier transform representation
 $$
    A_\lambda:=\chi(\sqrt{\Delta_g}-\lambda)+\chi(-\sqrt{\Delta_g}-\lambda) = \frac{2}{\sqrt{2\pi}}\int_\R\hat{\chi}(t)e^{-it\lambda}\cos(t\sqrt{\Delta_g})\,dt.
 $$
The operator $\chi(-\sqrt{\Delta_g}-\lambda)$ can be included in the remainder $R_\lambda$ since 
 \begin{equation}\label{eq:regularizing}
    \norm{\chi(-\sqrt{\Delta_g}-\lambda)(\cdot,\cdot)}_{L^\ii L^2(M\times M)}\le C,
 \end{equation}
 for some $C$ independent of $\lambda$. The bound \eqref{eq:regularizing} can be obtained as in the discussion following \cite[(3.2.13)]{Sogge-book-14}. It relies on rough $L^\ii$ bounds on eigenfunctions of $\Delta_g$ coming from Sobolev embeddings and the fact that the eigenvalues of $\Delta_g$ grow (at most) polynomially by Weyl's law; see \cite[Lem. 3.1.2]{Sogge-book-14}.
 
As explained, for instance, in \cite[Thm. 3.1.5]{Sogge-book-14} using the Hadamard parametrix for the fundamental solution of the wave equation, there is a $\tau>0$ such that for all $t\in[-\tau,\tau]$ we can decompose the wave propagator as
 $$
 \cos(t\sqrt{\Delta_g})=\tilde{K}_t + \tilde{R}_t,
 $$
 with $\tilde{R}_t(\cdot,\cdot)\in C^1_{t,x,y}([-\tau,\tau]\times M\times M)$ and with a (main part) $\tilde{K}_t$ described in detail below. We choose $\epsilon\le\tau$ and note that the integral kernel of
 $$
 \frac{2}{\sqrt{2\pi}}\int_\R\hat{\chi}(t)e^{-it\lambda}\tilde{R}_t\,dt
 $$
 is bounded in $L^\infty L^2(M\times M)$ uniformly in $\lambda$ and can therefore be included in $R_\lambda$.

 We now describe the integral kernel of the operator $\tilde{K}_t$ in local coordinates, as explained in \cite[Sec. 2.4]{Sogge-book-14}: locally around any fixed $x_0\in M$, there are $L\in\N$ and functions $(\alpha_\nu)_{\nu=1,\ldots,L}\subset C^\ii(M\times M)$ such that for all $t\in[0,\tau]$ we have 
 \begin{equation}\label{eq:first-decomposition}
    \tilde{K_t}(x,y)=\sum_{\nu=1}^L\alpha_\nu(x,y)\partial_tE_\nu(t,d_g(x,y)).
 \end{equation}
 Here $d_g(x,y)$ denotes the geodesic distance between $x$ and $y$ and the distributions $E_\nu$ are defined for instance in \cite[Lem. 17.4.2]{Hormander-book-III} or in \cite[Prop. 1.2.4]{Sogge-book-14}. As explained in \cite[Rem. 1.2.5]{Sogge-book-14}, each distribution is (modulo smooth functions, which can be absorbed into $\tilde{R}_t$) a finite linear combination of distributions of the form
 \begin{equation}\label{eq:fourier-Enu}
    (t,r)\in\R_+\times\R_+\mapsto t^\ell\int_{\R^N}e^{ir\xi_1\pm it|\xi|}|\xi|^{-k}d\xi,
 \end{equation}
 for $k,\ell\in\N$. Here, the distribution is understood in the sense of \cite[Thm. 0.5.1]{Sogge-book}
 
Next, let $\delta$ and $\epsilon$ be constants with $0<\delta<\epsilon\le\tau$. We claim that, if we restrict ourselves to $t$ with $t\in[\delta,\epsilon]$, then we may assume that the functions $\alpha_\nu$ in \eqref{eq:first-decomposition} are supported in $\{(x,y)\in M\times M:\ \delta/2 \le d_g(x,y)\le 2\epsilon \}$. In fact, the phase function $\xi\mapsto r\xi_1\pm t|\xi|$ is non-degenerate if $t\in[\delta,\epsilon]$ and $r\not\in (\delta/2,2\epsilon)$. Therefore, we have the kernel bound
 $$\norm{\1_{d_g(x,y)\notin(\delta/2,2\tau)} \int_{\R^N}e^{id_g(x,y)\xi_1\pm it|\xi|}|\xi|^{-k}\,d\xi\,dt}_{L^\ii_{x,y}(M\times M)}\le C_{k},$$
for $t\in[\delta,\epsilon]$, which again can be absorbed in $R_\lambda$. Therefore, by multiplying $\alpha$ with a smooth cut-off function, we may achieve the claimed support condition on $\alpha$.

We now consider a Schwartz function $\chi$ on $\R$ with Fourier transform which has a support in $[\delta,\epsilon]$. Integrating \eqref{eq:fourier-Enu} on $t$, the contribution of $K_t$ to the integral kernel of the operator $A_\lambda$ may thus be written as a finite linear combination of functions of the form
 \begin{equation}\label{eq:kernel-singular-part}
    \alpha(x,y)\int_{\R^N}e^{id_g(x,y)\xi_1}\chi^{(\ell)}(|\xi|-\lambda)|\xi|^{-k}\,d\xi
 \end{equation}
 with $\alpha\in C^\ii(M\times M)$ supported in $d_g(x,y)\in[\delta/2,2\epsilon]$, and $k,\ell\in\N$. Using the fast decay of $\chi^{(\ell)}$, we see that the contribution of the last integral for $|\xi|\notin[(1/C)\lambda,C\lambda]$ is fast decaying in $\lambda$ (in the space $L^\ii_{x,y}$ for instance), and it is thus enough to consider integral kernels of the form
 \begin{multline*}
    \alpha(x,y)\int_{\R^N}e^{id_g(x,y)\xi_1}\chi^{(\ell)}(|\xi|-\lambda)\beta(|\xi|/\lambda)|\xi|^{-k}\,d\xi \\
    = \frac{i^\ell\lambda^{N-k}}{\sqrt{2\pi}}\alpha(x,y)\int_\R\int_{\R^N}e^{i\lambda\Psi(x,y,t,\xi)}\hat{\chi}(t)\beta(|\xi|)t^\ell|\xi|^{-k}\,d\xi\,dt
 \end{multline*}
 for some $\beta\in C^\ii(\R)$ supported in $[1/C,C]$, for some $C>0$, and the phase function
 $$\Psi(x,y,t,\xi)=d_g(x,y)\xi_1+t(|\xi|-1).$$
 For $(x,y)$ fixed, the function $(t,\xi)\mapsto\Psi(x,y,t,\xi)$ has a unique critical point $(t,\xi)=(d_g(x,y),-e_1)$ which is non-degenerate (uniformly in $(x,y)$ in the considered region). By \cite[Cor. 1.1.8]{Sogge-book}, we may thus write 
 $$\alpha(x,y)\int_\R\int_{\R^N}e^{i\lambda\Psi(x,y,t,\xi)}\hat{\chi}(t)\beta(|\xi|)t^\ell|\xi|^{-k}\,d\xi\,dt=\lambda^{-\frac{N+1}{2}}e^{i\lambda\Psi(x,y,d_g(x,y),-e_1)}a_{k,\ell}(x,y,\lambda)$$
 with $a_{k,\ell}$ having the desired behaviour \eqref{eq:bounds-a}, and noticing that $\Psi(x,y,d_g(x,y),-e_1)=-d_g(x,y)$ concludes the sketch of proof of Lemma \ref{lem:parametrix}.

\section{WKB approximation}

\subsection{Spherical harmonics}

It is well known that
\begin{equation}
\label{eq:spherharm}
Y_\ell^m(\theta,\phi) = \sqrt{\frac{2\ell+1}{4\pi} \cdot \frac{(\ell-m)!}{(\ell+m)!}}\ P_{\ell}^m(\cos\theta) \ e^{im\phi} \,,
\end{equation}
where $P_{\ell}^m$ are the associated Legendre polynomials.

We will use the following few facts about these functions:
\begin{equation}
\label{eq:ylmnorm}
\int_{\Sph^2} |Y_\ell^m|^2\,d\omega = 1 \,,
\end{equation}
\begin{equation}
\label{eq:legparity}
(-1/2,1/2)\ni s\mapsto P_\ell^m(1/2+s) \ \text{is even/odd if}\ \ell+m \ \text{is even/odd} \,,
\end{equation}
\begin{equation}
\label{eq:legzero}
P^m_\ell(0)=(-1)^{\frac{\ell+m}2} \frac{2^m}{\sqrt{\pi}}\frac{\Gamma\left(\frac{\ell+m+1}{2}\right)}{\Gamma\left(\frac{\ell-m+2}{2}\right)} \,,
\end{equation}
\begin{equation}
\label{eq:legzeroder}
(P^m_\ell)'(0) = (-1)^{\frac{\ell+m-1}{2}} \frac{2^{m+1}}{\sqrt{\pi}}\frac{\Gamma\left(\frac{\ell+m+2}{2}\right)}{\Gamma\left(\frac{\ell-m+1}{2}\right)} \,.
\end{equation}
For \eqref{eq:legzero} and \eqref{eq:legzeroder} we refer to \cite[(8.6.1), (8.6.3)]{AbrSte-book}.

Let us derive the associated Legendre equation. The Laplacian in spherical coordinates reads 
$$\Delta_{\Sph^2}=-\frac{1}{\sin\theta}\partial_\theta \sin\theta\partial_\theta -\frac{1}{\sin^2\theta}\partial_\phi^2 \,.$$
Inserting the factorization \eqref{eq:ylmfactor} we see that the function $g^m_\ell$ satisfies the equation
\begin{equation}
\label{eq:glmeq}
-\frac{1}{\sin\theta}\frac{d}{d\theta}\left(\sin\theta\frac{d}{d\theta} g^m_\ell\right)+\frac{m^2}{\sin^2\theta}g^m_\ell=\ell(\ell+1)g^m_\ell,
\end{equation}
and \eqref{eq:ylmnorm} gives
\begin{equation}
\label{eq:glmnorm}
\int_0^\pi|g^m_\ell(\theta)|^2\sin\theta\,d\theta=\frac{1}{2\pi}.
\end{equation}
In order to bring the equation for $g_\ell^m$ into the standard form $-y''+Q(x)y=0$ we define $v_\ell^m$ by \eqref{eq:vlmdef} and we easily deduce the equation \eqref{eq:vlmeq} and the normalization \eqref{eq:vlmnorm}. We also note that
\begin{equation}
\label{eq:vlmleg}
v_\ell^m(\theta) = \sqrt{\frac{2\ell+1}{4\pi} \cdot \frac{(\ell-m)!}{(\ell+m)!}}\ \sqrt{\cos\theta}\ P_{\ell}^m(\sin\theta) \,, \qquad \theta\in[-\pi/2,\pi/2] \,.
\end{equation}

\subsection{Proof of Lemma \ref{lem:parameters-estimates}}

Here we prove the bounds on $Q_{\ell,m}(\theta)$ for $\theta\in I_\ell^{(\#)}$ claimed in Lemma \ref{lem:parameters-estimates}.

\begin{proof}[Proof of Lemma \ref{lem:parameters-estimates}]
We first consider the case $\#=2$. Define $k:=\ell-m$, and write 
$$Q_{\ell,m}(\theta)=\ell^2\left(\frac{1}{\cos^2\theta}-1\right)+\frac{k^2-2k\ell-\frac14}{\cos^2\theta}-\frac14-\ell.$$
Therefore,
$$
Q_{\ell,m}(\theta)\le\ell^2\left(\frac{1}{\cos^2\theta}-1\right)-\frac{k (2\ell - k)}{\cos^2\theta}.
$$
Let $\epsilon>0$. Then for $\ell$ large enough and for all $\theta\in I_\ell^{(2)}$,
$$\frac{1}{\cos^2\theta}-1\le (1+\epsilon) \theta^2\le (1+\epsilon) \eta_2^2r_\ell/\ell,$$
$$\frac{k(2\ell-k)}{\cos^2\theta}\ge k(2\ell-k) \ge 2(1-\epsilon) k\ell \ge 2(1-\epsilon)\ell r_\ell,$$
hence
$$Q_{\ell,m}(\theta)\le-(2(1-\epsilon)-(1+\epsilon)\eta_2^2)\ell r_\ell \,.$$
Since $\eta_2<\sqrt 2$, we can choose $\epsilon>0$ small enough so that $c_2:=2(1-\epsilon)-(1+\epsilon)\eta_2^2>0$. This is the claimed upper bound.

For the lower bound, we have similarly
$$Q_{\ell,m}(\theta)\ge\frac{-2k\ell-\frac14}{\cos^2\theta}-\frac14-\ell.$$
Let $\epsilon>0$ and note that for $\ell$ large enough and $\theta\in I_\ell^{(2)}$,
$$\frac{2k\ell}{\cos^2\theta}\le4(1+\epsilon)\ell r_\ell,\qquad \frac{1}{4\cos^2\theta}+\frac14+\ell\le\epsilon\ell r_\ell$$
This gives
$$Q_{\ell,m}(\theta)\ge-(4+5\epsilon)\ell r_\ell \,,$$
which is the claimed lower bound (for any fixed choice of $\epsilon$).

Assume now $\#=\ii$. For the upper bound we estimate
$$Q_{\ell,m}(\theta)\le \frac{m^2}{\cos^2\theta}-\ell^2.$$
For $\ell$ large enough and for all $\theta\in I_\ell^{(\ii)}$ we have 
$$\frac{1}{\cos^2\theta}=\frac{1}{\sin^2(\pi/2-|\theta|)}\le \frac{1}{\sin^2(\eta_1r_\ell/\ell)}
\le\frac{(1+\epsilon)\ell^2}{\eta_1^2r_\ell^2},$$
hence
$$Q_{\ell,m}(\theta)\le\frac{4(1+\epsilon)\ell^2r_\ell^2}{\eta_1^2r_\ell^2}-\ell^2=-\left(1-(1+\epsilon)\frac{4}{\eta_1^2}\right)\ell^2 \,.$$
Since $\eta_1>2$ we can choose $\epsilon>0$ small enough so that $c_2:= 1-(1+\epsilon)\frac{4}{\eta_1^2}>0$. This is the claimed upper bound.

For the lower bound, we use 
$$
Q_{\ell,m}(\theta)\ge-\frac{1}{4\cos^2\theta}-\frac14-\ell(\ell+1).
$$
Let $\epsilon>0$ and note that for $\ell$ large enough and $\theta\in I_\ell^{(\ii)}$,
$$
\frac{1}{\cos^2\theta}\le \frac{(1+\epsilon)\ell^2}{\eta_1^2r_\ell^2}\le \epsilon\ell^2,
$$
Thus,
$$
Q_{\ell,m}(\theta) \ge - \left( \frac{\epsilon}{4} + \frac{1}{4\ell^2} + 1 + \frac{1}{\ell} \right) \ell^2 \,,
$$
which is the claimed lower bound (for any fixed choice of $\epsilon$).
\end{proof}

\subsection{Reminder on the WKB approximation}

In order to prove Proposition \ref{wkbsingle} we will use the following version of the WKB approximation, which can be found, for instance, in \cite[Ch. 2, Sec. 2]{Fedoryuk-book}.

\begin{proposition}[WKB approximation]\label{prop:WKB-general}
 Let $a>0$, $I=(-a,a)\subset\R$, and $Q:I\to\R$ an even function of class $C^2$ which does not vanish anywhere on $I$. Define the functions
 $$S(x)=\int_0^x\sqrt{Q(t)}\,dt,\quad\tilde{y_1}(x)=\frac{e^{S(x)}}{Q(x)^{1/4}},$$
 $$\cE(x)=\int_0^{|x|}\left|\frac{1}{8Q(t)^{3/2}}\left(Q''(t)-\frac{5Q'(t)^2}{4Q(t)}\right)\right|\,dt,$$
 for all $x\in I$. Then, the unique solution $y_1:I\to\C$ of
 $$-y_1''+Q(x)y_1=0,\quad y_1(0)=\tilde{y_1}(0),\quad y_1'(0)=\tilde{y_1}'(0),$$
 satisfies the error bound
 $$\left|\frac{y_1}{\tilde{y_1}}-1\right|\le 2\left(e^{2\cE}-1\right).$$
 Furthermore, defining $y_2(x)=y_1(-x)$ and $\tilde{y_2}(x)=\tilde{y_1}(-x)$ for all $x\in I$, the function $y_2$ solves
 $$-y_2''+Q(x)y_2=0,\quad y_2(0)=\tilde{y_2}(0)=y_1(0),\quad y_2'(0)=\tilde{y_2}'(0)=-y_1'(0).$$
Moreover, the functions $(y_1,y_2)$ form a basis of solutions to the ODE $-y''+Q(x)y=0$ 
\end{proposition}

The fact that $(y_1,y_2)$ form a basis of solutions follows from the fact that $\tilde{y_1}'(0)\neq0$ which, in turn, follows from $Q'(0)=0\neq4Q(0)^{3/2}$.

\subsection{Proof of Proposition \ref{wkbsingle}}

In order to prove Proposition \ref{wkbsingle} we apply Proposition \ref{prop:WKB-general} with $Q=Q_{\ell,m}$ and the interval $I=I_\ell^{(\#)}$. Let us denote the corresponding remainder by $\mathcal E_{\ell,m}$. We now show that this remainder is indeed small.

\begin{lemma}[Accuracy of the WKB approximation]\label{lem:error-WKB}
 There exist $C>0$ and $L\ge1$ such that for all $\ell\ge L$ and for all $\theta\in I_\ell^{(\#)}$ we have
 $$\cE_{\ell,m}(\theta)\le Cr_\ell^{-1} \,.$$
\end{lemma}

\begin{proof}[Proof of Lemma \ref{lem:error-WKB}]
While the order of the error is the same in both cases, the proof in the two cases is different. Assume first that $\ell\ge L$ with $L$ large enough such that the conclusions of Lemma \ref{lem:parameters-estimates} holds. Let us start with the case $\#=2$. Since the function $V(\theta)=\cos^{-2}(\theta)$ satisfies $V'(0)=0$ and $V''(0)=2$, we may estimate for all $\theta\in I_\ell^{(2)}$ and all $\ell$ large enough
$$|Q_{\ell,m}'(\theta)|^2\le C_1\ell^4|\theta|^2\le C_1'\ell^3r_\ell,\quad|Q''(\theta)|\le C_2\ell^2.$$
Using the lower bound $|Q(\theta)|\ge c_2 \ell r_\ell$ obtained in Lemma \ref{lem:parameters-estimates}, we deduce the error estimate
 $$\cE_{\ell,m}(\theta)\le C(\ell r_\ell)^{-\frac32}\ell^2\left|\int_0^{\theta}dt\right|\le C(\ell r_\ell)^{-\frac32}\ell^2(r_\ell/\ell)^{\frac12},$$
 which is the desired estimate.

 Assume now $\#=\ii$. For all $\theta\in I_\ell^{(\ii)}$ we have
  $$Q_{\ell,m}'(\theta)=\left(m^2-\frac14\right)\frac{2\sin\theta}{\cos^3\theta},\qquad Q_{\ell,m}''(\theta)=\left(m^2-\frac14\right)\left[\frac{2}{\cos^2\theta}+\frac{6\sin^2\theta}{\cos^4\theta}\right],$$
hence for $\ell\ge1$,
$$|Q_{\ell,m}'(\theta)|\le\frac{8r_\ell^2}{\cos^3\theta},\qquad|Q_{\ell,m}''(\theta)|\le\frac{24r_\ell^2}{\cos^4\theta}.$$ 
Using the lower bound $|Q|\ge c_2\ell^2/2$ from Lemma \ref{lem:parameters-estimates}, we may thus estimate
$$\cE_{\ell,m}(\theta)\le C\ell^{-3}\left(r_\ell^2\int_0^{\pi/2-\eta_1r_\ell/\ell}\frac{dt}{\cos^4t}+\ell^{-2}r_\ell^4\int_0^{\pi/2-\eta_1r_\ell/\ell}\frac{dt}{\cos^6t}\right).$$
We have for $\ell$ large enough
$$\int_0^{\pi/2-\eta_1r_\ell/\ell}\frac{dt}{\cos^4t}=\int_{\eta_1r_\ell/\ell}^{\pi/2}\frac{dt}{\sin^4t}\le C\int_{\eta_1r_\ell/\ell}^{\pi/2}\frac{dt}{t^4}\le C'(\ell/r_\ell)^3,$$
and by the same argument
 $$\int_0^{\pi/2-\eta_1r_\ell/\ell}\frac{dt}{\cos^6t}\le C(\ell/r_\ell)^{5},$$
leading to the desired estimate.
\end{proof}

\begin{proof}[Proof of Proposition \ref{wkbsingle}]
Let us introduce the functions
$$
\mathcal C_\ell^m := \frac{\cos S_{\ell,m}}{|Q_{\ell,m}|^{1/4}} \,,
\qquad
\mathcal S_\ell^m := \frac{\sin S_{\ell,m}}{|Q_{\ell,m}|^{1/4}} \,.
$$
Moreover, let $y_1,y_2,\tilde y_1, \tilde y_2$ be the functions introduced in Proposition \ref{prop:WKB-general}. We note that
$$
\tilde y_1 + \tilde y_2 = 2 e^{-i\pi/4} \mathcal C_\ell^m \,,
\qquad
\tilde y_1 - \tilde y_2 = 2i e^{-i\pi/4} \mathcal S_\ell^m \,.
$$
We first assume that $\ell+m$ is even. Since $\mathcal C_\ell^m$ is even, $\tilde y_1+\tilde y_2$ is so as well and therefore $(y_1+y_2)'(0) = (\tilde y_1 +\tilde y_2)'(0)=0$. Moreover, by \eqref{eq:legparity} and \eqref{eq:vlmleg} we know that $v_\ell^m$ is even and therefore $(v_\ell^m)'(0)=0$. Since $y_1$ and $y_2$ are a basis of solutions, we conclude that
$$
v_\ell^m = c_{\ell,m} 2^{-1} e^{i\pi/4} (y_1+y_2)
$$
where
\begin{equation}
\label{eq:clmdef}
c_{\ell,m} = 2 e^{-i\pi/4} \frac{v_\ell^m(0)}{y_1(0)+y_2(0)} = 2 e^{-i\pi/4}  \frac{v_\ell^m(0)}{\tilde y_1(0)+\tilde y_2(0)} = \frac{v_\ell^m(0)}{\mathcal C_\ell^m(0)} \,.
\end{equation}
Thus, by Proposition \ref{prop:WKB-general} and Lemma \ref{lem:error-WKB},
\begin{align*}
\left| v_\ell^m - c_{\ell,m} \mathcal C_\ell^m \right|
& = 2^{-1} |c_{\ell,m}| \left| (y_1+y_2) -(\tilde y_1+\tilde y_2) \right| \\
& \le 2^{-1} |c_{\ell,m}| \left( \left|y_1 - \tilde y_1\right| + \left|y_2 - \tilde y_2\right| \right) \\
& \le |c_{\ell,m}| \left( |y_1| + |y_2| \right) \left( e^{2\mathcal E_{\ell,m}} - 1 \right) \\
& \le C r_\ell^{-1} |c_{\ell,m}| |Q_{\ell,m}|^{-1/4} \,,
\end{align*}
which is the claimed bound for $\ell+m$ even. The proof for $\ell+m$ odd is similar. We only record the formula
\begin{equation}
\label{eq:clmdefodd}
c_{\ell,m} = \frac{(v_\ell^m)'(0)}{(\mathcal S_\ell^m)'(0)} \,.
\end{equation}
and omit the details.
\end{proof}

\subsection{The constants $c_{\ell,m}$}

\begin{proof}[Proof of Lemma \ref{lem:normalization-constant}]
Assume that $\ell+m$ is even. Then, according to \eqref{eq:clmdef}, \eqref{eq:vlmleg} and \eqref{eq:legzero}
\begin{align*}
c_{\ell,m} & = |Q_{\ell,m}(0)|^{1/4}\ v^m_\ell(0) = |Q_{\ell,m}(0)|^{1/4} \sqrt{\frac{2\ell+1}{4\pi} \ \frac{(\ell-m)!}{(\ell+m)!}}\ P_\ell^m(0) \\
& = |Q_{\ell,m}(0)|^{1/4} \ \sqrt{ \frac{2\ell+1}{4\pi} \ \frac{(\ell-m)!}{(\ell+m)!} } \ (-1)^{\frac{\ell+m}{2}}\ \frac{2^{m}}{\sqrt\pi} \ \frac{\Gamma\left(\frac{\ell+m+1}{2}\right)}{\Gamma\left(\frac{\ell-m+2}{2}\right)} \,.
\end{align*}
Using Stirling's approximation,
$$
\ln\Gamma(x) = x\ln x-x + \frac{1}{2}\ln\frac{2\pi}{x} + O(\frac{1}{x}) \,,
$$
it is easy to compute that
$$
|c_{\ell,m}| = |Q_{\ell,m}(0)|^{1/4} \ \sqrt{ \frac{2\ell+1}{2\pi^2}} \ \left( \ell+m+1 \right) ^{-1/4}\left( \ell-m+1 \right)^{-1/4} \left( 1+ \mathcal O(\frac{1}{\ell-m}) \right) \,.
$$
This, together with the upper and lower bounds on $Q_{\ell,m}(0)$ from Lemma \ref{lem:parameters-estimates}, implies that $|c_{\ell,m}|^2$ is bounded from above and from below by a positive constant times $\ell$ in both cases $\#=2$ and $\#=\ii$.

The proof in the case $\ell+m$ odd is similar, using \eqref{eq:legzeroder} instead of \eqref{eq:legzero}, and is omitted.
\end{proof}

%%%%%%%%%%%%%%%%%%%%%%%%%%%%%%%%%%%%%%%%%%%%

\section{A Kato--Seiler--Simon inequality on manifold}

We recall that the Kato--Seiler--Simon inequality \cite[Thm. 4.1]{Sim} on $\R^N$ implies that for $2\le p \le\infty$,
$$
\norm{ \beta(\sqrt{\Delta}) W }_{\gS^{p}(L^2(\R^N))} \le (2\pi)^{-N/p} \norm{W}_{L^p(\R^N)} \left( |\Sph^{N-1}| \int_0^\infty |\beta(\lambda)|^p \lambda^{N-1}\,d\lambda \right)^{1/p} \,.
$$
In this appendix we prove the following generalization to manifolds.

\begin{proposition}\label{kss}
Let $2\le p \le \infty$. Then
$$
\norm{ \beta(\sqrt{\Delta_g}) W }_{\gS^{p}(L^2(M))} \le C^{1/p} \norm{W}_{L^p(M)} \left(  \sum_{n=0}^\infty \sup_{n\leq\lambda\leq n+1} |\beta(\lambda)|^p (1+n)^{N-1} \right)^{1/p}
$$
where $C$ depends only on $M$.
\end{proposition}

\begin{proof}
The inequality for $p=\infty$ is immediate and we assume in the following that $2\le p<\infty$. According to the Lieb--Thirring inequality \cite{LieThi-76} (see also \cite[Cor. 8.2]{Sim}) we have
\begin{align*}
\norm{ \beta(\sqrt{\Delta_g}) W }_{\gS^p(L^2(M))}^p & \le \tr |\beta(\sqrt{\Delta_g})|^p |W|^p = \sum_{n=0}^\infty \tr |\beta(\sqrt{\Delta_g})|^p \Pi_n |W|^p \\
& \le \sum_{n=0}^\infty \sup_{n\leq\lambda\leq n+1} |\beta(\lambda)|^p \tr \Pi_n |W|^p
\end{align*}
As in the proof of Theorem \ref{thm:main} one can bound, using the pointwise Weyl law,
$$
\tr \Pi_n |W|^p \le C (1+n)^{N-1} \norm{W}_{L^p(M)}^p
$$
with a constant $C$ depending only on $M$. We conclude that
$$
\norm{ \beta(\sqrt{\Delta_g}) W }_{\gS^p(L^2(M))}^p \le C \norm{W}_{L^p(M)}^p \sum_{n=0}^\infty \sup_{n\leq\lambda\leq n+1} |\beta(\lambda)|^p (1+n)^{N-1} \,,
$$
which is the claimed inequality.
\end{proof}

We emphasize again that, for the special choice $\beta(\tau)=\1(\lambda\le\tau\le\lambda+1)$ and for $2<p<\infty$, Theorem \ref{thm:main} gives stronger results than Proposition \ref{kss} (in the sense that in a worse Schatten space one obtains a better growth in $\lambda$ for $W$ in a fixed $L^p$ space).

% %%%%%%%%%%%%%%%%%%%%%%%%%%%%%%%%%%%%%%%%%%
% %%%%%%%%%%%%%%%%%%%%%%%%%%%%%%%%%%%%%%%%%%
% \bibliographystyle{siam}
% \bibliography{biblio}

\end{document}